\documentclass[reqno,a4paper,oneside]{amsart}

\usepackage{amsmath, amsthm, amscd, amsfonts, amssymb, graphicx, color,mathrsfs,mathtools,enumerate}
\usepackage{lineno}
\usepackage{dsfont}
\modulolinenumbers[5]

\usepackage{bm}
\usepackage{graphicx}%
\usepackage{multirow}%
\usepackage{amsmath,amssymb,amsfonts}%
\usepackage{amsthm}%
\usepackage{mathrsfs}%
\usepackage[title]{appendix}%
\usepackage{xcolor}%
\usepackage{textcomp}%
\usepackage{manyfoot}%
\usepackage{booktabs}%
\usepackage{algorithm}%
\usepackage{algorithmicx}%
\usepackage{algpseudocode}%
\usepackage{listings}%

\usepackage[latin1,utf8]{inputenc}
\usepackage[english]{babel}
\usepackage{eucal,amsfonts,amssymb,amsmath,amsthm,epsfig,mathrsfs}
\usepackage{cancel,soul}
\usepackage{color}
\usepackage{amscd,amsxtra}
\usepackage{enumerate}
\usepackage{latexsym}
\usepackage{bm}
\usepackage{setspace}
\usepackage{textcomp}

\newtheorem{thm}{Theorem}[section]
\newtheorem{lemma}[thm]{Lemma}
\newtheorem{prop}[thm]{Proposition}
\newtheorem{coro}[thm]{Corollary}

\newtheorem{hyp1}[thm]{Hypotheses}

\newtheorem{rmk}[thm]{Remark}

\newtheorem{defn}[thm]{Definition}

\newcommand{\N}{\mathbb N}
\newcommand{\R}{\mathbb R}
\newcommand{\X}{\mathcal{X}}
\newcommand{\K}{\mathcal{K}}

\newcommand{\Id}{{\operatorname{I}}}
\newcommand{\norm}[1]{{\left\|#1\right\|}}
\newcommand{\scal}[2]{{\left\langle #1,#2\right\rangle}}

\makeatletter
\@namedef{subjclassname@2020}{%
  \textup{2020} Mathematics Subject Classification}
\makeatother

\begin{document}

\frenchspacing

\title[Pathwise uniqueness in infinite dimension under weak structure condition]{Pathwise uniqueness in infinite dimension under weak structure condition}

\author[D. Addona and D. A. Bignamini]{Davide Addona, Davide A. Bignamini}

\address{D. Addona: Dipartimento di Scienze Matematiche, Fisiche e Informatiche, Universit\`a degli studi di Parma, Parco Area delle Scienze 53/a (Campus), 43124 PARMA, Italy}
\email{\textcolor[rgb]{0.00,0.00,0.84}{da.bignamini@uninsubria.it}}

\address{D. A. Bignamini: Dipartimento di Scienza e Alta Tecnologia (DISAT), Universit\`a degli Studi dell'In\-su\-bria, Via Valleggio 11, 22100 COMO, Italy}
\email{\textcolor[rgb]{0.00,0.00,0.84}{da.bignamini@uninsubria.it}}


\begin{abstract}
Let $U,H$ be two separable Hilbert spaces and $T>0$. We consider a stochastic differential equation which evolves in the Hilbert space $H$ of the form
\begin{align}
\label{SDEa}
dX(t)=AX(t)dt+\widetilde{\mathscr L}B(X(t))dt+GdW(t), \quad t\in[0,T], \quad X(0)=x \in H,    
\end{align}
where $A:D(A)\subseteq H\to H$ is the infinitesimal generator of a strongly continuous semigroup $(e^{tA})_{t\geq0}$, $W=(W(t))_{t\geq0}$ is a $U$-cylindrical Wiener process defined on a normal filtered probability space $(\Omega,\mathcal{F},\{\mathcal{F}_t\}_{t\in [0,T]},\mathbb{P})$, $B:H\to H$ is a bounded and $\theta$-H\"older continuous function, for some suitable $\theta\in(0,1)$, and  $\widetilde{\mathscr L}:H\to H$ and $G:U\to H$ are linear bounded operators. We prove that, under suitable assumptions on the coefficients, the weak mild solution to equation \eqref{SDEa} depends on the initial datum in a Lipschitz way.
This implies that, for \eqref{SDEa}, pathwise uniqueness holds true. Here, the presence of the operator $\widetilde{\mathscr L}$ plays a crucial role. In particular the conditions assumed on the coefficients cover the stochastic damped wave equation in dimension $1$ and the stochastic damped Euler--Bernoulli Beam equation up to dimension $3$ even in the hyperbolic case.
\end{abstract}





\maketitle

\section{Introduction}
Let $U,H$ be two separable Hilbert spaces and $T>0$. We consider a SDE which evolves in the Hilbert space $H$ of the form
\begin{align}
\label{SDE}
dX(t)=AX(t)dt+\widetilde{\mathscr L}B(X(t))dt+GdW(t), \quad t\in[0,T], \quad X(0)=x \in H,    
\end{align}
where $A:D(A)\subseteq H\to H$ is the infinitesimal generator of a strongly continuous semigroup $(e^{tA})_{t\geq0}$, $W=(W(t))_{t\geq0}$ is a $U$-cylindrical Wiener process defined on a normal filtered probability space $(\Omega,\mathcal{F},\{\mathcal{F}_t\}_{t\in [0,T]},\mathbb{P})$, $B:H\to H$ is a bounded and $\theta$-H\"older continuous function, for some suitable $\theta\in(0,1)$, and  $\widetilde{\mathscr L}:H\to H$ and $G:U\to H$ are linear bounded operators.

We assume that for every $T>0$ and every $x\in H$ equation \eqref{SDE} admits a unique weak (mild) solution $X=(X(t,x))_{t\in[0,T]}$ which satisfies, for every $t\in[0,T]$,
\begin{align}
\label{mild_sol}
X(t,x)=e^{tA}x+\int_0^t e^{(t-s)A}\widetilde{\mathscr L} B(X(s,x))ds+\int_0^t e^{(t-s)A}GdW(s), \qquad \mathbb P\textup{-a.s.}    
\end{align}
We are investigating the pathwise uniqueness of the weak mild solution $X$ to the SDE given by \eqref{SDE}. Specifically, we aim to extend the results in \cite{AddBig, Dap-Fla2014} to cases where $A$ is not necessarily the generator of an analytic semigroup. We underline that if $(e^{tA})_{t\geq 0}$ is a compact semigroup then the existence of a weak mild solution to \eqref{SDE} is guaranteed, see Remark \ref{Weak}.
To be more precise, we prove a stronger result which implies that pathwise uniqueness holds for \eqref{SDE}. Indeed, we show that there exists a positive constant $C_T$ such that for every $x,y\in H$, if $X(,\cdot,x)$ and $X(\cdot,y)$ are weak mild solutions to \eqref{SDE}, with initial datum $x$ and $y$, respectively, defined on the same probability space $(\Omega,\mathcal F, \mathbb P)$ with respect to the same cylindrical Wiener process $W$, then
\begin{equation}
\label{Heat-lip}
\sup_{t\in [0,T]}\mathbb{E}\left[\norm{X(t,x)-X(t,y)}_H\right]\leq C_T\norm{x-y}_H,\qquad x,y\in H.
\end{equation}

Furthermore, we generalize the results in \cite{AddMasPri23} by removing the assumption that $G:U\rightarrow H$ is a Hilbert-Schmidt operator and weaken the so called {\it structure condition} $\widetilde{\mathscr L}=G$. The novelty of this paper allows us to establish pathwise uniqueness for the stochastic hyperbolic damped wave equation, a case not covered in the results of \cite{AddBig,AddMasPri23,Dap-Fla2010}.  Moreover exploiting  the result in \cite{On04} we also establish strong existence to \eqref{SDE}, see Corollary \ref{Strong}.

Before specifically describing the results in this paper, we provide a brief overview of the literature concerning pathwise uniqueness.

In the recent literature, the problem of pathwise uniqueness for stochastic evolution equations of the form of \eqref{SDE} has been extensively explored, see, for instance, \cite{Cer-Dap-Fla2013, Dap-Fla2010, Dap-Fla2014, DPFPR13, DPFPR15, DPFRV16, MasPri17, MasPri23, Pri2021, Zam00}. One of the primary reasons for this interest lies in the relationship between pathwise uniqueness and existence of strong solutions to SDEs, shown in \cite{Ya-Wa71} in finite dimension and generalized for stochastic differential equations with values in $2$-smooth Banach spaces in \cite{On04}. 
The Zvonkin transformation and the It\^o-Tanaka trick are two fundamental tools to obtain pathwise uniqueness for SDEs like \eqref{SDE}, see for instance \cite{FlGuPr10, Kr-Ro05,Ver,Zv74}. Here, we exploit a modification of the It\^o-Tanaka trick, similar to that introduced in \cite{AddBig}. 

In the following subsections, we provide a detailed comparison between the results presented in this paper and those already known in the literature.

\subsubsection{Stochastic damped wave equation}
Let $T>0$. We consider the following stochastic semilinear damped wave equation
\begin{align}
\label{eq_damped}
\left\{
\begin{array}{ll}
\displaystyle \frac{\partial^2y}{\partial t^2}(t)
= y''(t)-\rho \left(-D_{xx}^2\right)^{\alpha}\left(\frac{\partial y}{\partial t}(t)\right)+\left(-D_{xx}^2\right)^{-\sigma}c\left(\cdot, y(t),\frac{\partial y}{\partial t} (t)\right)+(\left(-D_{xx}^2\right)^{\alpha})^{-\gamma}\dot{W}(t), 
\vspace{1mm} \\
y(0)=y_0 \in L^2(0,\pi), & \vspace{1mm} \\
\displaystyle \frac{\partial y}{\partial t}(0)=y_1 \in L^2(0,\pi),
\end{array}
\right.
\end{align}
where $t\in [0,T]$, $D_{xx}^2$ is the realization of the second-order derivative in $L^2(0,\pi)$ with Dirichlet homogeneous boundary conditions, $\{W(t)\}_{t\geq 0}$ is a $L^2(0,\pi)$-cylindrical Wiener process, $c:[0,\pi]\times \R^2\to \R$ is a measurable function, for a.e.$ \xi\in[0,\pi]$ the function $c(\xi,\cdot,\cdot):\R^2\to \R$ is continuous, there exist $c_1\in L^\infty([0,\pi])$, $c_2\in L^2([0,\pi])$ and $\theta\in(0,1)$ such that $|c(\xi,x)-c(\xi,y)|\leq c_1(\xi)|x-y|^\theta$ and $|c(\xi,x)|\leq c_2(\xi)$ for every $x,y\in \R$ and a.e. $\xi\in[0,\pi]$, $\alpha\in(0,1)$ and $\rho,\sigma,\gamma\geq0$ are nonnegative parameters. 
Assume that one of the following sets of assumptions is verified:
\begin{itemize}
\item[\rm(i)]  $\alpha\in\left(0,\frac12\right]$, $\gamma\in\left(\frac14-\frac\alpha2,\frac\alpha2\right)$, $\theta\in\left(\frac23\cdot\frac{\gamma+\alpha}{\alpha},1\right)$ and $\sigma>\frac12-\alpha$;
\item[\rm(ii)] $\alpha\in\left[\frac12,1\right)$, $\gamma\in\left(\frac14-\frac\alpha2,\frac12-\frac\alpha2\right)$, $\theta\in\left(\frac23\cdot \frac{\gamma+1-\alpha}{1-\alpha},1\right)$ if $\gamma+2\alpha<\frac32$, $\theta\in\left(\frac{4\gamma+2\alpha-1}{2\gamma+\alpha},1\right)$ if $\gamma+2\alpha\geq \frac32$, and $\sigma>0$.
\end{itemize}
If further $\rho\neq 2^{2-2\alpha}n^{1-2\alpha}$ for every $n\in\N$, then, for mild solutions to  \eqref{eq_damped}, inequality \eqref{Heat-lip} holds true with $H=L^2(0,\pi)\times L^2(0,\pi)$ and $X=\left((-D^2_{xx})^{-\frac12}y,\frac{\partial}{\partial t}y\right)$, see Corollary \ref{dampedalfa}.  We emphasize that case (i) is not addressed by the results in \cite{AddBig,AddMasPri23} and, to the best of our knowledge, represents a novel contribution to the literature. Specifically,  in \cite{AddBig} this case is not covered, since for $\alpha<\frac12$ the operator associated to the damped wave equation does not generate an analytic semigroup. This gap is filled thanks to the presence of the operator $(-D_{xx}^2)^{-\sigma}$, which allows to recover the needed regularity. On the other hand, in \cite{AddMasPri23}, the assumption that in \eqref{SDE} the operator $G$ is a Hilbert–Schmidt operator implies $\gamma>\frac14$ in \eqref{eq_damped} and so $\alpha> \frac12$.

We stress that, under assumptions (ii) (with $\sigma=0$), pathwise uniqueness for \eqref{eq_damped} has been proved in \cite{AddBig}. However, the techniques exploited in \cite{AddBig} do not allow to obtain a Lipschitz dependence on the initial data, as estimate \eqref{Heat-lip} implies.

\subsubsection{Stochastic damped Euler--Bernoulli Beam equation}
Let $T>0$. We focus on the following stochastic semilinear damped Euler-Bernoulli beam equation
   \begin{align}
\label{Eulero-Beam}
\left\{
\begin{array}{ll}
\displaystyle \frac{\partial^2y}{\partial t^2}(t)
= -\Delta^2 y(t)-\rho (\Delta^2)^{\alpha}\left(\frac{\partial y}{\partial t}(t)\right)+(\Delta^2)^{-\sigma}c\left(\cdot, y(t),\frac{\partial y}{\partial t}(t)\right)+(\Delta^2)^{-\gamma}\dot{W}(t), \ t\in(0,T],
\vspace{1mm} \\
y(0)=y_0\in L^2([0,2\pi]^m), & \vspace{1mm} \\
\displaystyle \frac{\partial y}{\partial t}(0)=y_1 \in L^2([0,2\pi]^m),
\end{array}
\right.
\end{align}
where $\alpha\in (0,1)$, $\rho,\gamma,\sigma\geq 0$, $c:[0,\pi]\times \R^2\rightarrow\R$ is as above and $\Delta$ is the realization of the Laplace operator with periodic boundary conditions in $L^2([0,2\pi]^m)$ with $m=1,2,3$.

In \cite[Example 6.1]{AddMasPri23}, pathwise uniqueness for \eqref{Eulero-Beam} is established when $m=1$, $\sigma=\gamma$ and $(\Delta^{-2})^\gamma$ is a trace-class operator and it is proved that estimate \eqref{Heat-lip} holds true with $H=L^2([0,2\pi])\times L^2([0,2\pi])$. In this paper, we gain the same result removing the condition that $\Delta^{-2\gamma}$ is a trace class operator, $\sigma$ is not necessarily equal to $\gamma$ and $m$ can be $2$ or $3$ (see Corollary \ref{coro:damped_euler_beam} for the precise statement). Moreover, we show that if $\Delta^{-2\gamma}$ is a trace class operator then the following estimate, stronger than \eqref{Heat-lip}, holds true:
\begin{equation*}
\mathbb{E}\bigg[\sup_{t\in [0,T]}\norm{X(t,x)-X(t,y)}_H\bigg]\leq C_T\norm{x-y}_H,\qquad x,y\in H.
\end{equation*}

\subsubsection{Stochastic heat equation}
Let $T>0$ and consider the SPDE
\begin{align}\label{Heat}
\left\{
\begin{array}{ll}
\displaystyle  dX(t)=\Delta X(t)dt+(-\Delta)^{-\sigma}B(X(t))dt+(-\Delta)^{-\gamma}dW(t), \quad t\in[0,T],  \vspace{1mm} \\
X(0)=x\in H,
\end{array}
\right.
\end{align}
where $\gamma,\sigma\geq 0$, $B\in C_b^\theta(H;H)$ and $\Delta$ is the realization of the Laplace operator with periodic boundary conditions in $H=L^2([0,2\pi]^m)$ with $m=1,2,3$. Under suitable assumptions on $\gamma$, pathwise uniqueness for \eqref{Heat} is established in \cite{AddBig, Dap-Fla2014}. Also in \cite{AddMasPri23} pathwise uniqueness for \eqref{Heat} is proved in the case $\sigma=\gamma$ and the Lipschitz estimate \eqref{Heat-lip} is provided. The main result of this paper shows that \eqref{Heat-lip} holds true even if $\sigma<\gamma$, in particular, if $m=1,2$ then it is sufficient that $\sigma>0$.

\section{Notations}
\label{sec:notation}
 Let $\K$ be a Banach space endowed with the norm $\norm{\cdot}_\K$. We denote by $\mathcal{B}(\K)$ the Borel $\sigma$-algebra associated to the norm topology in $\K$. 

Let $(\Omega,\mathcal{F},\mathbb{P})$ be a probability space and let $\xi:(\Omega,\mathcal{F},\mathbb{P}) \rightarrow (\K,\mathcal{B}(\K))$ be a random variable. We denote by 
\[
\mathbb{E}[\xi]:=\int_\Omega \xi(w)\mathbb{P}(d\omega)=\int_\K x[\mathbb{P}\circ\xi^{-1}](dx)
\]
the expectation of $\xi$ with respect to $\mathbb{P}$. Let $\{Y(t)\}_{t\geq 0}$ be a $\K$-valued stochastic process defined on a normal filtered probability space $(\Omega,\mathcal{F},\{\mathcal{F}_t\}_{t\geq 0},\mathbb{P})$. We say that $\{Y(t)\}_{t\geq 0}$ is pathwise continuous $a.s.$ (almost surely) if there exists $\Omega_0\in\mathcal{F}$ such that $\mathbb{P}(\Omega_0)=1$ and for every $\omega\in\Omega_0$ the function $t\rightarrow Y(t)(\omega)$ is continuous. 

Let $\X$ be a separable Hilbert space and let $\{g_k\}_{k\in\N}$ be an orthonormal basis of $\X$. We call $\X$-cylindrical Wiener process a stochastic process $\{W(t)\}_{t\geq 0}$ defined on a normal filtered probability space $(\Omega,\mathcal{F},\{\mathcal{F}_t\}_{t\geq 0},\mathbb{P})$, such that 
\[
W(t):=\sum_{k=1}^{\infty} \beta_k(t)g_k \qquad \forall t\geq0,
\]
where $\{\beta_n(t)\}_{t\geq 0}$, $n\in\N$, are real independent Brownian motions on $(\Omega,\mathcal{F},\{\mathcal F_t\}_{t\geq0},\mathbb{P})$. We recall the definitions of weak and strong existence for \eqref{SDE}.

\begin{defn}\label{weak-solution}
A weak (mild) solution to \eqref{SDE} is a couple $(X,W)$ where $W=\{W(t)\}_{t\in [0,T]}$ is a $U$-cylindrical Wiener process defined on a filtered probability space $(\Omega,\mathcal{F},\{\mathcal{F}_t\}_{t\in [0,T]},\mathbb{P})$ and $X=\{X(t,x)\}_{t\in[0,T]}$ is a $H$-valued $\{\mathcal{F}_t\}_{t\in [0,T]}$-adapted stochastic process such that for every $t\in [0,T]$ 
\begin{align*}
X(t,x)=e^{tA}x+\int_0^te^{(t-s)A}B(X(s,x))ds+W_A(t),\qquad \mathbb{P}{\rm -a.s.},    
\end{align*}
where $\{W_A(t)\}_{t\geq 0}$ is the stochastic convolution process given by 
\begin{equation*}
W_A(t):=\int^t_0e^{(t-s)A}GdW(s),\qquad \mathbb P{\rm-a.s.}
\end{equation*}
for every $t\geq0$.\\
Moreover, we say that the strong existence holds true for \eqref{SDE}, if for for every $U$-cylindrical Wiener process $W=\{W(t)\}_{t\in [0,T]}$ defined on a filtered probability space $(\Omega,\mathcal{F},\{\mathcal{F}_t\}_{t\in [0,T]},\mathbb{P})$, there exists a process $X=\{X(t,x)\}_{t\in[0,T]}$ such that $(W,X)$ is a weak mild solution to \eqref{SDE}.
\end{defn}

Let $\K_1$ and $\K_2$ be two real separable Banach spaces equipped with the norms $\norm{\cdot}_{\K_1}$ and $\norm{\cdot}_{\K_2}$, respectively.
We denote by $\Id_{\K_1}$ the identity operator on $\K_1$. Let $n\in\N$.
We denote by $\mathcal{L}^{(n)}(\K_1;\K_2)$ the set of multilinear continuous maps from $\K^n_1$ to $\K_2$, if $n=1$ we write $\mathcal{L}(\K_1;\K_2)$, and if $\K_1=\K_2$ we write $\mathcal{L}^{(n)}(\K_1)$.

We denote by $B_b(\K_1;\K_2)$ the set of bounded and Borel measurable functions from $\K_1$ into $\K_2$. If $\K_2=\R$, then we simply write $B_b(\K_1)$. We denote by $C_b(\K_1;\K_2)$ ($UC_b(\K_1;\K_2)$, respectively) the space of bounded and continuous (uniformly continuous, respectively) functions from $\K_1$ into $\K_2$. If $\K_2=\R$, then we simply write $C_b(\K_1)$ and $UC_b(\K_1)$. We endowed $C_b(\K_1;\K_2)$ and $UC_b(\K_1;\K_2)$ with the norm
\[
\norm{f}_{\infty}=\sup_{x\in \K_1}\|f(x)\|_{\K_2}.
\]

We denote by $C_b^\alpha(\K_1;\K_2)$ the subspace of $C_b(\K_1;\K_2)$ of the $\alpha$-H\"older continuous functions. The space $C_b^\alpha(\K_1;\K_2)$ is a Banach space if endowed with the norm
\begin{align*}
\norm{f}_{C_b^\alpha(\K_1;\K_2)}:=\norm{f}_\infty+[f]_{C_b^\alpha(\mathcal{K}_1;\mathcal{K}_2)},
\end{align*}
where $[\cdot]_{C_b^\alpha(\mathcal{K}_1;\mathcal{K}_2)}$ denotes the standard seminorm on $C_b^\alpha(\mathcal{K}_1;\mathcal{K}_2)$.
If $\K_2=\R$, then we simply write $C_b^\alpha(\K_1)$.

Let $k\in\N$ and let $f:\K_1\rightarrow \K_2$ be a $k$-times Fr\'echet differentiable function. We denote by $D^k f(x)$ its Fr\'echet derivative of order $k$ at $x\in\K_1$. 
For $k\in\N$, we denote by $C_b^{k}(\K_1;\K_2)$ ($UC^k_b(\K_1;\K_2)$, respectively) the space of bounded, uniformly continuous and $k$ times Fr\'echet differentiable functions $f:\K_1\rightarrow\K_2$ such that $D^if\in C_b(\K_1;\mathcal{L}^{(i)}(\K_1;\K_2))$ ($D^if\in UC_b(\K_1;\mathcal{L}^{(i)}(\K_1;\K_2))$, respectively), for $i=1,\ldots,k$. We endow $C_b^{k}(\K_1;K_2)$ and $UC_b^{k}(\K_1;K_2)$  with the norm
\[
\norm{f}_{C_b^{k}(\K_1;\K_2)}:=\norm{f}_{\infty}+\sum^k_{i=1}\sup_{x\in\K_1}\|D^if(x)\|_{\mathcal{L}^{(i)}(\K_1;\K_2)}.
\]
We denote by $C_b^{0,2}([0,T]\times\K_1;\K_2)$ the space of functions $f:[0,T]\times\K_1\rightarrow\K_2$ such that $f(t,\cdot)\in C^2_b(\K_1;\K_2)$ for any $t\in [0,T]$ and $f(\cdot,x)\in C([0,T];\K_2)$ for any $x\in \K_1$. If $\K_2=\R$, then we simply write $C_b^{0,2}([0,T]\times\K_1)$. 

Let $\X$ be a separable Hilbert space equipped with the inner product $\scal{\cdot}{\cdot}_{\X}$.
If $f\in C^2_b(\X)$, then we denote by $\nabla f (x)$ and $\nabla^2f(x)$ the Fr\'echet gradient and Hessian at $x\in\X$, respectively. We say that $Q\in\mathcal{L}(\X)$ is \emph{non-negative} (\emph{positive}) if for every $x\in \X\backslash\{0\}$
\[
\langle Qx,x\rangle_\X\geq 0\ (>0).
\]
On the other hand, $Q \in \mathcal{L}(\X)$ is a \emph{non-positive} (respectively, \emph{negative}) operator if $-Q$ is non-negative (respectively, positive). Let $Q\in\mathcal{L}(\X)$ be a non-negative and self-adjoint operator. We say that $Q$ is a trace-class operator if
\begin{align}\label{trace_defn}
{\rm Trace}_{\mathcal X}[Q]:=\sum_{n=1}^{+\infty}\langle Qe_n,e_n\rangle_\X<\infty
\end{align}
for some (and hence for all) orthonormal basis $\{e_n:n\in\N\}$ of $\X$. We recall that the definition of trace-class operator given in \eqref{trace_defn}, is independent of the choice of the orthonormal basis. 
Let $\mathcal Y$ be another separable Hilbert space and let $R\in\mathcal{L}(\X;\mathcal Y)$. We say that $R$ is a Hilbert-Schmidt operator if 
\[
\norm{R}^2_{\mathcal L_2(\X;\mathcal Y)}:=\sum^{+\infty}_{k=1}\norm{Rg_k}^2_{\mathcal Y}<\infty,
\]
for some (and hence for all) orthonormal basis $\{g_k:k\in\N\}$ of $\X$. It follows that if $R$ is a Hilbert-Schmidt operator, then $RR^*$ and $R^*R$ are trace-class operators and
\begin{align*}
{\rm Trace}_{\mathcal Y}[RR^*]={\rm Trace}_{\mathcal X}[R^*R]=\|R\|^2_{\mathcal L_2(\mathcal X;\mathcal Y)}.    
\end{align*}

\section{Abstract results}
In this section, we will prove the main results of this paper. We will exploit an approach based on the finite-dimensional approximations previously introduced in \cite{AddBig}. In Subsection \ref{preliminari}, we state the main assumption of the paper and provide some preliminary result, while Subsection \ref{proof} is devoted to prove the main results of this paper. 

\subsection{Finite dimensional approximation and preliminary results}\label{preliminari}
The following assumptions allow us to exploit the approximation technique of \cite{AddBig}.

\begin{hyp1}\label{hyp:finito-dimensionale}
The following conditions hold true.
\begin{enumerate}[\rm(i)]
\item $A:{\rm Dom}(A)\subseteq H\rightarrow H$
is the infinitesimal generator of a strongly continuous semigroup.

\item $G:U\rightarrow H$ is a bounded linear operator, such that
\begin{equation}\label{Gtilda}
G=\widetilde{G}\mathcal{V}
\end{equation}
where $\widetilde{G}\in\mathcal{L}(U;H)$ and $\mathcal{V}\in \mathcal{L}(U)$.

\item $B:H\rightarrow H$ is a bounded and $\theta$-H\"older continuous function with $\theta\in (0,1)$ such that
\begin{equation}\label{Btilda}
B=\widetilde{G}\widetilde{B},
\end{equation}
where $\widetilde{B}\in C_b^\theta(H;U)$.

\item $\widetilde{\mathscr L}: H\rightarrow H$ is a linear bounded operator and there exists $\mathcal{K}\in \mathcal{L}(U)$ such that 
\begin{equation}\label{commutazione}
\widetilde{\mathscr L}\widetilde{G}=\widetilde{G}\mathcal{K}.
\end{equation}
\item There exists $\eta\in (0,1)$ such that for every $t>0$ we have
\[
\int^t_0\frac{1}{s^\eta}{\rm Trace}_H\left[e^{sA}GG^*e^{sA^*}\right]ds<\infty.
\]

\item \label{Accan}There exists a sequence of finite-dimensional subspaces $\{H_n\}_{n\in\N}\subseteq H$ such that $H=\overline{\cup_{n\in\N}H_n}$, $H_0:=\{0\}$ and for every $n\in\N$ we have
\begin{align*}
& H_{n-1}\subseteq H_{n}, \qquad 
H_{n-1}\subseteq {\rm Dom}(A),\qquad A(H_{n}\backslash H_{n-1})\subseteq \left(H_{n}\backslash H_{n-1}\right)\cup \{0\}.
\end{align*}
We further assume that $\widetilde{\mathscr L}$ and the orthogonal projection on $H_n$ commute.

\item\label{contrin} For every $t>0$ we have
\begin{align}
   &e^{tA}(H)\subseteq Q^{\frac12}_t(H),\qquad\qquad \qquad\qquad\quad Q_{t}:=\int^t_0e^{sA}GG^*e^{sA^*}ds;\label{contron}\\
&\int^t_0\norm{\Gamma_s}^{1-\theta}_{\mathcal{L}(H)}\|\Gamma_s\widetilde{G}\|_{\mathcal{L}(U;H)}ds<\infty, \quad\qquad\Gamma_t:=Q^{-\frac12}_te^{tA}\label{supercontron},
\end{align}
where $s\mapsto \norm{\Gamma_s}^{1-\theta}_{\mathcal{L}(H)}$ and $\mapsto\|\Gamma_s\widetilde{G}\|_{\mathcal{L}(U;H)}$ are bounded from below functions in $(0,t)$ for every $t>0$. Further, we assume that there exists $\theta'<\theta$ such that
\begin{align*}
\int^t_0\norm{\Gamma_s}^{1-\theta'}_{\mathcal{L}(H)}ds<\infty.  
\end{align*}
\end{enumerate}
\end{hyp1}

\begin{rmk}
Let $\mathcal{O}$ be a suitable bounded subset of $\R^d$ and let $\Delta$ be the realization of the Laplace operator in $L^2(\mathcal{O})$ with suitable boundary conditions. In the applications of Section \ref{Examples} we choose $\mathcal{K}=(-\Delta)^{-\sigma}$ and $\mathcal{V}=(-\Delta)^{-\gamma}$ for some $\sigma>0$ and $\gamma\geq 0$. In particular, when we consider perturbed versions of the eat equation we can set $U=H$ and $\widetilde{G}=\Id_H$, see  Subsection \ref{Heatcase}. For more details on our assumptions see \cite[Remark 3.5]{AddBig}.
\end{rmk}

\begin{rmk}\label{Weak}
We refer to \cite{Cho-Gol1995,Kun2013,Kun2013-2} for general assumptions which guarantee the weak well-posedness of \eqref{SDE}. In particular, \cite[Proposition 3]{Cho-Gol1995} guarantees that if $(e^{tA})_{t\geq 0}$ is a compact semigroup then there exists a weak mild solution to \eqref{SDE}. For a discussion on weak uniqueness we refer to \cite{BerOrrSca2024, Pri2021} for the stochastic heat equation with singular drift and to \cite{Han2024} for the stochastic wave equation with multiplicative noise.
\end{rmk}

For every $n\in\N$, we set
\begin{align}\label{coefficienti-n}
A_n:=AP_n,\quad \widetilde{G}_n:=P_n \widetilde{G},\quad G_n=\widetilde{G}_n\mathcal{V}, \quad \widetilde{B}_n=\widetilde{B}\circ P_n,\quad B_n=\widetilde{G}_n\widetilde{B}_n,\quad \widetilde{\mathscr L}_n=P_n\widetilde{\mathscr L},
\end{align}
by \eqref{commutazione} and Hypotheses \ref{hyp:finito-dimensionale}\eqref{Accan}, for every $n\in\N$ we have
\begin{equation}\label{commutazione-n}
\widetilde{\mathscr L}_n\widetilde{G}_n=\widetilde{G}_n\mathcal{K}.
\end{equation}
We recall the following facts:
\begin{align}
& e^{tA}_{|H_n}=e^{tA_n} \qquad \forall t\geq0,\label{P5}\\
&\norm{\widetilde{\mathscr L}_n}_{\mathcal{L}(H)}\leq \norm{\widetilde{\mathscr L}}_{\mathcal{L}(H)} \qquad \forall n\in\N, \label{stimaopn}\\
&\norm{B_n}_{C^\theta_b(H_n;H_n)}\leq \norm{B}_{C^\theta_b(H;H)} \qquad \forall n\in\N.
\label{holdrBn}
\end{align}
Fix $T>0$, $x\in H$ and let $\{X(t,x)\}_{t\in [0,T]}$ be a weak mild solution to \eqref{SDE}.
For every $n\in\N$ we introduce the process $X_n=(X_n(t,x))_{t\in[0,T]}$ defined, for every $t\in[0,T]$, as
\begin{align}
\label{mild_sol_n}
X_n(t,x)=e^{tA_n}P_nx+\int_0^t e^{(t-s)A_n}\widetilde{\mathscr L}_n B_n(X(s,x))ds+W_{A,n}(t),\qquad \mathbb P{\rm -a.s.,}
\end{align}
where 
\begin{equation}\label{csn}
W_{A,n}(t)=\int_0^te^{tA_n}G_ndW(s).
\end{equation}
In particular, $X_n$ fulfills
\begin{align}
\label{mild_sol_n_var}
& dX_n(t,x)=A_nX_n(t,x)dt+\widetilde{\mathscr L}_nB_n(X(t,x))dt+G_ndW(t), \quad t\in[0,T], \qquad  X_n(0,x)=P_nx. \end{align}

\begin{prop}
Assume that Hypotheses \ref{hyp:finito-dimensionale} hold true. Therefore, for every fixed $T>0$ and $x\in H$ we have
\begin{equation}\label{supE}
\lim_{n \to \infty}\sup_{t\in [0,T]}\mathbb{E}\left[\|X_{n}(t,x)-X(t,x)\|_H^2\right]dt=0.
\end{equation} 
In addition, if $G\in\mathcal L_2(U;H)$, then for every fixed $T>0$ and $x\in H$ we have
\begin{equation}\label{Esup}
\lim_{n \to \infty}\mathbb{E}\Big[\sup_{t\in [0,T]}\|X_{n}(t,x)-X(t,x)\|_H^2\Big]dt=0.
\end{equation} 
\end{prop}

\begin{proof}
The first part of the statement has been already proved in \cite[Proposition 4.2]{AddBig}. Here, we show that \eqref{Esup} holds true. Fix $T>0$ and $x\in H$. We begin by proving that 
\begin{align}\label{convsupWA}
\lim_{n\rightarrow +\infty}\mathbb{E}\left[\sup_{t\in [0,T]}\|W_A(t)-W_{A_n}(t)\|_H^2\right]=0.
\end{align}
From the definition of $W_{A,n}$ and \eqref{P5}, for every $t\geq0$ we can write
\begin{align*}
W_A(t)-W_{A,n}(t)=\int_0^t e^{(t-s)A}\left(G-P_nG\right)dW(s), \qquad \mathbb P{\rm-a.s.}
\end{align*}
By the Kotelentz inequality for submartingales (see \cite{Hau-Sei2001,Kot1984}), there exists $C>0$ such that for every $n\in\N$ we get
\begin{align*}
\mathbb{E}\left[\sup_{t\in [0,T]}\|W_A(t)-W_{A_n}(t)\|^2_H\right]
&\leq C\int_0^T \left\|\left(I-P_n\right)G\right\|_{\mathcal{L}_2(U;H)}^2ds.
\end{align*}
Hence,
\begin{align}
\lim_{n\rightarrow \infty}\mathbb{E}\left[\sup_{t\in [0,T]}\|W_A(t)-W_{A_n}(t)\|^2_H\right]
&\leq C\lim_{n\rightarrow \infty}\int_0^T \norm{(\Id-P_n)Gg_k}_H^2ds\notag\\
&= CT\!\lim_{n\rightarrow \infty}\sum_{k=1}^\infty\norm{(\Id-P_n)Gg_k}_H^2,\label{lim_conv_stoc}
\end{align}
where $\{g_k\}_{n\in\N}$ is an orthonormal basis of $U$. Since $G\in \mathcal{L}_2(U;H)$, for every $k\in\N$ we get
\begin{equation*}
\norm{(\Id-P_n)Gg_k}_H^2\leq \norm{Gg_k}_H^2, \qquad \lim_{n\rightarrow\infty}\norm{(\Id-P_n)Gg_k}_H^2=0, \qquad \sum_{k=1}^{\infty}\norm{Gg_k}_H^2=\norm{G}_{\mathcal{L}_2(U;H)}^2.
\end{equation*}
By the dominated convergence theorem, we infer that
\begin{align}
\label{conv_somm_G}
\lim_{n\to\infty}  \left\|\left(I-P_n\right)G\right\|_{\mathcal{L}_2(U;H)}^2=\lim_{n\to\infty}\sum_{k=1}^{\infty}\norm{(\Id-P_n)Gg_k}_H^2=0.  
\end{align}
Combining \eqref{lim_conv_stoc} and \eqref{conv_somm_G}, we obtain \eqref{convsupWA}.

From the definition of $B_n$, Hypotheses \ref{hyp:finito-dimensionale} and \eqref{P5}, for every $n\in\N$ we have
\begin{align*}
&\norm{\int_0^te^{(t-s)A_n}\widetilde{\mathscr L}_nB_n(X(s,x))ds-\int_0^te^{(t-s)A}\widetilde{\mathscr L}B(X(s,x))ds}^2_H\\
&\leq\sup_{t\in[0,T]}\norm{e^{tA}}^2_{\mathcal{L}(H)}\int_0^T\norm{\widetilde{\mathscr L}_nB_n(X(s,x))-\widetilde{\mathscr L}B(X(s,x))}^2_Hds,\quad \mathbb{P}-{\rm a.s.}
\end{align*}
Since $\norm{\widetilde{\mathscr L}_nB_n}_\infty\leq \|\widetilde{\mathscr L}\|_{\mathcal L(H)}\norm{B}_\infty<\infty$ and noticing that, for every $n\in\N$ and $s\in[0,T]$,
\begin{align*}
& \norm{\widetilde{\mathscr L}_nB_n(X(s,x))-\widetilde{\mathscr L}B(X(s,x))}_H \\
& \leq \norm{\widetilde{\mathscr L}_nB_n(X(s,x))-\widetilde{\mathscr L}_nB(X(s,x))}_H+\norm{\widetilde{\mathscr L}_nB(X(s,x))-\widetilde{\mathscr L}B(X(s,x))}_H \\
& \leq \|\widetilde{\mathscr L}\|_{\mathcal L(H)}\norm{B_n(X(s,x))-B(X(s,x))}_H+\norm{\widetilde{\mathscr L}_nB(X(s,x))-\widetilde{\mathscr L}B(X(s,x))}_H \\
& \leq \|\widetilde{\mathscr L}\|_{\mathcal L(H)} \big(\|P_nB(P_n X(s,x))-P_nB(X(s,x))\|_H +\|P_nB(X(s,x))-B(X(s,x))\|_H \big)\\
& + \norm{\widetilde{\mathscr L}_nB(X(s,x))-\widetilde{\mathscr L}B(X(s,x))}_H \\
& \leq\|\widetilde{\mathscr L}\|_{H} [B]_{C_b^\theta(H;H)}\|P_nX(s,x)-X(s,x)\|^\theta+\|P_nB(X(s,x))-B(X(s,x))\|_H \\
& +\norm{\widetilde{\mathscr L}_nB(X(s,x))-\widetilde{\mathscr L}B(X(s,x))}_H,
\end{align*}
by applying the dominated convergence theorem we infer
\begin{align}
\notag 
&\lim_{n\rightarrow \infty}\mathbb{E}\left[\sup_{t\in [0,T]}\norm{\int_0^te^{(t-s)A_n}\widetilde{\mathscr L}_nB_n(X(s,x))ds-\int_0^te^{(t-s)A}\widetilde{\mathscr L}B(X(s,x))ds}^2_H\right] \\
& \leq \sup_{t\in[0,T]}\norm{e^{tA}}^2_{\mathcal{L}(H)}
\lim_{n\to\infty} \mathbb E \left[\int_0^T\norm{\widetilde{\mathscr L}_nB_n(X(s,x))-\widetilde{\mathscr L}B(X(s,x))}^2_Hds\right]=0.
\label{convsupint}
\end{align}
Finally, by \eqref{P5} we have 
\begin{equation}
\label{conveta}
\lim_{n\to\infty}\sup_{t\in[0,T]}\norm{e^{tA_n} P_nx-e^{tA}x}_H\leq \left(\sup_{t\in [0,T]}\norm{e^{tA}}_{\mathcal{L}(H)}\right)\lim_{n\to\infty}\norm{P_nx-x}_H.
\end{equation}
Hence, \eqref{convsupWA}, \eqref{convsupint} and \eqref{conveta} yield the statement.
\end{proof}

Let $n\in\N$. We consider the following integral equality 
\begin{equation}\label{Back-Kolmo-intro}
    U_n(t,x)=\int^T_t \mathcal{R}_n(r-t)\left( DU_n(r,\cdot)\widetilde{\mathscr L}_nB_n(\cdot)+B_n(\cdot)\right)(x)dr ,\quad t\in [0,T], \ x\in H_n,
\end{equation}
where $\{\mathcal{R}_n(t)\}_{t\geq 0}$ is the vector-valued Ornstein-Uhlenbeck semigroup associated to \eqref{mild_sol_n_var} with $B_n\equiv 0$. By \eqref{Btilda} (o ci vuole \eqref{commutazione}?) and \eqref{commutazione-n} we get 
\begin{align*}
U_n(t,x)&=\int^T_t \mathcal{R}_n(r-t)\left( DU_n(r,\cdot)\widetilde{\mathscr L}_nB_n(\cdot)+B_n(\cdot)\right)(x)dr \\
&=\int^T_t \mathcal{R}_n(r-t)\left( DU_n(r,\cdot)\widetilde{G}_n\mathcal{K}\widetilde{B}_n(\cdot)+B_n(\cdot)\right)(x)dr,
\end{align*}
hence for every $n\in\N$ by \cite[Proposition 4.5]{AddBig} (with $\widetilde{B}_n$ replaced by $\mathcal{K}\widetilde{B}_n$ and $F_n=B_n$) equation \eqref{Back-Kolmo-intro} admits a unique solution $U_n\in C_b^{0,1}([0,T]\times H_n;H_n)$ such that the map $x\rightarrow DU_n(t,x)\widetilde{G}_n$ belongs to $C^1_b(H_n;\mathcal{L}(U;H_n))$ for every $t\in [0,T]$. Moreover, for every $n\in\N$ and $t\in [0,T]$ we have
\begin{align}\label{stima-n}
\sup_{t\in[0,T]}\left(\norm{U_n(t,\cdot)}_{C^1_b(H_n;H_n)}+\|DU_n(t,\cdot)\widetilde{G}_n\|_{C^1_b(H_n;\mathcal{L}(U;H_n))}\right)&\leq M_{T} \|B_n\|_{C^\theta_b(H_n;H_n)},
 \end{align}
where $M_T$ is a positive constant such that
\[
\lim_{T\rightarrow0}M_T=0.
\]
For every $v\in H_n$ the function $U^v_{n}=\scal{U_n}{v}_H$ is the unique solution in $C^{0,1}_b([0,T]\times H_n)$ to the integral equation
\begin{equation}\label{Back-Kolmonk}
  U^v_{n}(t,x)=\int^T_t R_n(r-t)\left( \scal{\widetilde{G}_n^*\nabla U^v_{n}(r,\cdot)}{\mathcal{K}\widetilde{B}_n(\cdot)}_H+\scal{B_{n}}{v}_H\right)(x)dr ,\quad t\in [0,T], \ x\in H_n,
\end{equation}
and 
\begin{align}\label{stima-n-scalare}
\sup_{t\in[0,T]}\left(\norm{U^v_n(t,\cdot)}_{C^1_b(H_n)}+\|\widetilde{G}_n^*\nabla U_n(t,\cdot)\|_{C^1_b(H_n;U)}\right)\leq M_T\|\langle F_n,v\rangle_H\|_{C^\theta_b(H_n)},\quad
t\in[0,T].
\end{align}
Let $\{g_k:k\in\N\}$ be an orthonormal basis of $H$ such that 
\[
H_n={\rm span}\{g_1,...,g_{s_n}\}\quad \forall n\in\N,
\]
where $s_n:={\rm Dim}(H_n)$. The functions $U_{n,k}=\scal{U_n}{g_k}_H$, with $k=1,..., s_n$ verify
\begin{equation}\label{serieUk}
U_n=\sum^{s_n}_{k=1}U_{n,k}g_k,
\end{equation}
and, if $v,w\in H_n$ then
\begin{equation}\label{scambio}
\scal{DU_n(x)\widetilde{G}_nv}{w}=\scal{\widetilde{G}_n^*\nabla U^w_n(x)}{v},\quad x\in H_n,\; n\in\N.
\end{equation}
For more details about these Kolmogorov equations we refer to \cite[Appendix A]{AddBig} (see also \cite[Remark 4.6]{AddBig})

\begin{prop}
Assume that Hypotheses \ref{hyp:finito-dimensionale} hold true. For every $T>0$, $n\in\N$ and $x\in H$ the stochastic process $\{X_n(t,x)\}_{t\in [0,T]}$ defined in \eqref{mild_sol_n} satisfies the following equality: for every $t\in[0,T]$, 
\begin{align}
X_n(t,x)
& = e^{tA_n}P_nx-\widetilde{\mathscr L}_nU_n(t,X_n(t,x))+\widetilde{\mathscr L}_ne^{tA_n}U_n(0,P_nx) \notag \\
& -A_n\int_0^te^{(t-s)A_n}\widetilde{\mathscr L}_nU_n(s,X_n(s,x))ds \notag \\
& +\int_0^t e^{(t-s)A_n}\widetilde{\mathscr L}_n(B_n(X(s,x))-B_n(s,X_n(s,x)))ds \notag \\
& +\int_0^t e^{(t-s)A_n}\widetilde{\mathscr L}_n D_xU_n(s,X_n(s,x))(\widetilde{\mathscr L}_nB_n(X(s,x))-\widetilde{\mathscr L}_nB_n(X_n(s,x)))ds \notag \\
& +\int_0^t e^{(t-s)A_n}G_ndW(s)+ \int_0^t e^{(t-s)A_n}\widetilde{\mathscr L}_nD_xU_n(s,X_n(s,x))(G_ndW(s)), \qquad \mathbb P\textup{-a.s.}
\label{mild_sol_n_pert_ibp}
\end{align}
\end{prop}

\begin{proof}
Let $\{g_k:k\in\N\}$ be an orthonormal basis of $H$ such that 
\[
H_n={\rm span}\{g_1,...,g_{s_n}\}\quad \forall n\in\N,
\]
where $s_n:={\rm Dim}(H_n)$. Let $n\in\N$, $k=1,\ldots,s_n$ and let $U_{n,k}$ be the solution to \eqref{Back-Kolmonk} (with $v=g_k$), using the It\^o formula and the same approximation argument of \cite[Section 4.2]{AddBig}, we obtain
\begin{align*}
dU_{n,k}(t,X_n(t,x))&= \scal{\nabla_x U_{n,k}(t,X_n(t,x))}{\widetilde{\mathscr L}_nB_n(X(t,x))-\widetilde{\mathscr L}_nB_n(X_n(t,x))}dt\\
&-\scal{B_n(X_n(t,x))}{g_k}dt+  \scal{\nabla_x U_{n,k}(t,X_n(t,x)) }{G_ndW(t)}.
\end{align*}
Hence
\begin{align}
\scal{B_n(X_n(t,x))}{g_k}dt&=-dU_{n,k}(t,X_n(t,x))dt\notag\\
&+\scal{\nabla_xU_{n,k}(t,X_n(t,x))}{\widetilde{\mathscr L}_nB_n(X(t,x))-\widetilde{\mathscr L}_nB_n(X_n(t,x))}dt\notag\\
&+\scal{\nabla_x U_{n,k}(t,X_n(t,x))}{G_ndW(t)}.\label{Bnk}
\end{align}
Summing up $k$ from $1$ to $s_n$ in \eqref{Bnk}, by \eqref{serieUk} and \eqref{scambio} we obtain
\begin{align}\label{Bn}
B_n(X_n(t,x))dt&=-dU_{n}(t,X_n(t,x))\notag\\
&+D_xU_{n}(t,X_n(t,x))\left(\widetilde{\mathscr L}_nB_n(X(t,x))-\widetilde{\mathscr L}_nB_n(X_n(t,x))\right)dt\notag\\
&+D_x U_n(t,X_n(t,x))G_ndW(t).
\end{align}
By applying $\widetilde{\mathscr L}_n$ to both the sides of \eqref{Bn}, we infer that
\begin{align}\label{espr_Lambda_n_B_n}
\widetilde{\mathscr L}_nB_n(X_n(t,x))dt&=-d\widetilde{\mathscr L}_nU_{n}(t,X_n(t,x))\notag\\
&+\widetilde{\mathscr L}_nD_xU_{n}(t,X_n(t,x))\left(\widetilde{\mathscr L}_nB_n(X(t,x))-\widetilde{\mathscr L}_nB_n(X_n(t,x))\right)dt\notag\\
&+\widetilde{\mathscr L}_nD_x U_n(t,X_n(t,x))G_ndW(t) \qquad t\in[0,T].
\end{align}
Adding and subtracting $\widetilde{\mathscr L}_nB_n(t,X_n(t,x))dt$ in \eqref{mild_sol_n_var}, from \eqref{espr_Lambda_n_B_n} it follows that, for every $t\in[0,T]$,
\begin{align}
dX_n(t,x)
= & A_nX_n(t,x)dt+\widetilde{\mathscr L}_n(B_n(X(t,x))-B_n(t,X_n(t,x)))dt \notag \\
& +\widetilde{\mathscr L}_n D_xU_n(t,X_n(t,x))(\widetilde{\mathscr L}_nB_n(X(t,x))-\widetilde{\mathscr L}_nB_n(X_n(t,x)))dt \notag \\
& - d\widetilde{\mathscr L}_n U_n(t,X_n(t,x))
+\widetilde{\mathscr L}_nD_xU_n(t,X_n(t,x))(G_ndW(t))+G_ndW(t). 
\label{mild_sol_n_var_pert}
\end{align}
This implies that, for every $t\in[0,T]$,
\begin{align}
X_n(t,x)
& = e^{tA_n}P_nx+\int_0^t e^{(t-s)A_n}\widetilde{\mathscr L}_n(B_n(X(s,x))-B_n(s,X_n(s,x)))ds \notag \\
& +\int_0^t e^{(t-s)A_n}\widetilde{\mathscr L}_n D_xU_n(s,X_n(s,x))(\widetilde{\mathscr L}_nB_n(X(s,x))-\widetilde{\mathscr L}_nB_n(X_n(s,x)))ds \notag \\
& -\int_0^t e^{(t-s)A_n}d\widetilde{\mathscr L}_n U_n(s,X_n(s,x)) +\int_0^t e^{(t-s)A_n}G_ndW(s) \notag \\
& +\int_0^t e^{(t-s)A_n}\widetilde{\mathscr L}_nD_xU_n(s,X_n(s,x))(G_ndW(s)), \qquad \mathbb P\textup{-a.s.}
\label{mild_sol_n_pert}
\end{align}
Recalling that $ X_n(0,x)=P_nx$ and integrating by parts the third integral in the right-hand side of \eqref{mild_sol_n_pert}, we obtain \eqref{mild_sol_n_pert_ibp}.
\end{proof}

\subsection{Proof of the main results}
\label{proof}
The results of this subsection will be proved under the following additional conditions. If necessary, in the following hypotheses we consider (without changing the notation) the complexification of $H$ and we take our assumptions on the complexified space.

\begin{hyp1}\label{hyp:goal-addo}
\begin{enumerate}[\rm(i)]
\item
There exist $\beta\in [0,1)$ and, for every $T>0$, a positive constant $K_T>0$ such that 
\begin{align*}
\sup_{n\in\N} \| A_n e^{tA_n}\widetilde{\mathscr L}_n\|_{\mathcal{L}(H_n)}\leq K_T t^{-\beta}, \qquad t\in(0,T].   
\end{align*}
\item there exists a family of normalized (but not necessarily orthogonal) vectors $\{f_n:n\in\N\}$ of $H$ consisting of eigenvectors of $A^*$ such that $H=\overline{{\rm span}\{f_n:n\in\N\}}$ and there exists a sequence $(d_n)_{n\in\N}\subseteq \N$ such that 
\begin{enumerate}[\rm(a)]
\item for every $n\in\N$ we have
\begin{align*}    
\{f_1,\ldots,f_{s_n}\}=\bigcup_{i=1}^n \{e_1^i,\ldots,e_{d_i}^i\},\quad s_n=d_1+\ldots + d_n,
\end{align*}
where $\{e^i_n\,:\, i,n\in\N\}$ is a family of vectors such that for every $i,j\in\N$ with $i\neq j$, we have
\[
\langle e^i_k,e^j_h\rangle_H=0,\qquad k=1,\ldots,d_i, \ h=1,\ldots,d_j.
\]
\item There exists $d\in\N$ such that $d_n\leq d$ for every $n\in\N$.

\item The finite dimensional vector space $K_n:=H_{n}\setminus H_{n-1}\cup\{0\}$ is invariant for $\widetilde{\mathscr L}$ (and so also for $\widetilde{\mathscr L}^*$) segue dalla commutativit\`a tra $P_n$ e $\widetilde{\mathscr L}$? and we set $\zeta_n:=\|\widetilde{\mathscr L}\|_{\mathcal L(K_n)}=\|\widetilde{\mathscr L}^*\|_{\mathcal L(K_n)}$.  

\item For every $n\in\N$ and $j=1,\ldots,d_n$, the eigenvalue $\rho^n_j$ of $A^*$ associated to the eigenvector $e^n_j$ has negative real part. Moreover,
\begin{align}
\label{conv_serie_holder}
-\sum_{n\in\N}\zeta_n^2\sum_{j=1}^{d_n}\frac{\|B^n_j\|_{C_b^\theta(H)}^2}{{\rm Re}(\rho^n_j)}<\infty,   
\end{align}
where $B^n_j(\cdot)=\langle B(\cdot),e^n_j\rangle_H$.
\end{enumerate}
\end{enumerate}
\end{hyp1}

\begin{rmk}
\label{rmk:cont_dec_lambda_n}
\begin{enumerate}[\rm (i)]
\item If for every $n\in\N$ and every $i\in \{1,\ldots,d_n\}$, we set $\zeta_{n,i,j}:=\langle \widetilde{\mathscr L}^* e_i^n,e_j^n\rangle_H$, it follows that $|\zeta_{n,i,j}|\leq \zeta_n$ for every $n\in\N$ and every $i,j\in\{1,\ldots,d_n\}$. 
\item If $\{e^n_j: n,j\in\N\}$ are also eigenvectors of $\widetilde{\mathscr L}$ then we can replace condition \eqref{conv_serie_holder} with
\begin{align}
\label{conv_serie_holder_var}
-\sum_{n\in\N}\sum_{j=1}^{d_n}|\zeta_j^n|^2\frac{\|B^n_j\|_{C_b^\theta(H)}^2}{{\rm Re}(\rho^n_j)}<\infty,   
\end{align}
where $\zeta_j^n$ is the eigenvalue of $\widetilde{\mathscr L}$ associated to the eigenvector $e^n_j$.
\end{enumerate}
\end{rmk}

\begin{thm}\label{pathwiseuniqueness}
Assume that Hypotheses \ref{hyp:finito-dimensionale} and \ref{hyp:goal-addo} hold true. Then, for every $T>0$ and $x\in H$ pathwise uniqueness holds true for equation \eqref{SDE}. Moreover, for every $T>0$ there exists $C_T>0$ such that for every $x_1,x_2\in H$, if $X(\cdot,x_1)$, $X(\cdot,x_2)$ are two weak mild solutions defined on the same filtered probability space with respect to the same cylindrical Wiener process $W$, then
\begin{equation}\label{lip-scarsa}
    \sup_{t\in [0,T]}\mathbb{E}\left[\|X(t,x_1)-X(t,x_2)\|^2_H\right]\leq C_T\norm{x_1-x_2}^2_H.
\end{equation}
\end{thm}
\begin{proof}
Let $x_1,x_2\in H$ and let us set $X_i:=X(\cdot,x_i)$, $i=1,2$. We put
\[
\Delta(t):=\mathbb{E}\left[\norm{X_1(t)-X_2(t)}^2_H\right], \qquad t\in[0,T].
\]
By \eqref{mild_sol_n_pert_ibp}, for every $t\in [0,T]$ and every $n\in\N$ we have
\begin{equation}\label{unicitaL2n}
\Delta_n(t):=\mathbb{E}\left[\norm{X_{1,n}(t)-X_{2,n}(t)}^2_H\right]\leq 8\sum_{i=1}^8I_{i},
\end{equation}
where $\{X_{1,n}(t)\}_{t\in[0,T]}$ and $\{X_{2,n}(t)\}_{t\in[0,T]}$ are the processes defined in \eqref{mild_sol_n}, with $X(\cdot,x)$ replaced by $X_1$ and $X_2$, and 
\begin{align}
I_1
& :=\norm{e^{tA_n}P_nx_1-e^{tA_n}P_nx_2}^2_H+ \|\widetilde{\mathscr{L}}_ne^{tA_n}(U_n(0,P_nx_1)-U_n(0,P_nx_2))\|_H^2, \label{I1}\\
I_2
& :=\mathbb{E}\left[\|\widetilde{\mathscr L}_nU_n(t,X_{1,n}(t))-\widetilde{\mathscr L}_nU_n(t,X_{2,n}(t))\|^2_H\right] ,\label{I2}\\
I_3
& :=\mathbb{E}\left[\norm{A_n\int^t_0e^{(t-s)A_n}\widetilde{\mathscr L}_n\left(U_{n}(s,X_{1,n}(s))-U_{n}(s,X_{2,n}(s))\right)ds}^2_H\right],\label{I3}\\    I_{4+i}
& :=\mathbb{E}\left[\norm{\int^t_0e^{(t-s)A_n} \widetilde{\mathscr L}_n\left(B_n(X_{i}(s))-B_n(X_{i,n}(s))\right)ds}^2_H \right], \qquad i=1,2, \label{I4}\\    
I_{6+i}
& :=\mathbb{E}\left[\norm{\int^t_0e^{(t-s)A_n}\widetilde{\mathscr L}_n DU_{n}(s,X_{i,n}(s))\widetilde{\mathscr L}_n\left(B_n(X_{i}(s))-B_n(X_{i,n}(s))\right)ds}^2_H\right], \qquad i=1,2, \label{I6}\\
I_8
& :=\mathbb{E}\left[\norm{\int_0^te^{(t-s)A_n}\widetilde{\mathscr L}_n(DU_n(t,X_{1,n}(s))-DU_n(t,X_{2,n}(s)))G_ndW(s)}^2_H\right]\label{I8}.
\end{align}
    
Let us estimate $I_1$. By \eqref{P5}, \eqref{stimaopn}, \eqref{holdrBn} and \eqref{stima-n} we obtain
\begin{equation}\label{SI1}
I_1\leq \sup_{t\in[0,T]}\norm{e^{tA}}_{\mathcal{L}(H)}\left(1+M^2_{T}  \|\widetilde{\mathscr L}\|^2_{\mathcal{L}(H)}\|B\|^2_{C^\theta_b(H;H)}\right)\norm{x_1-x_2}^2_H.
\end{equation}

As far as $I_2$ is concerned, from  \eqref{stimaopn}, \eqref{holdrBn} and \eqref{stima-n} it follows that
\begin{equation}\label{SI2}
    I_2\leq M^2_{T}  \|\widetilde{\mathscr L}\|^2_{\mathcal{L}(H)}\|B\|^2_{C^\theta_b(H;H)}\Delta_n(t).
\end{equation}

To estimate $I_3$, we take advantage of  Hypothesis \ref{hyp:goal-addo}(i) and \eqref{stima-n} to infer that
\begin{align}\label{SI3}
I_3&\leq K^2_T\mathbb{E}\left[\left(\int^t_0 (t-s)^{-\beta}\norm{\left(U_{n}(s,X_{1,n}(s))-U_{n}(s,X_{2,n}(s))\right)}_Hds\right)^2\right]\notag\\
&\leq K^2_TM^2_T\norm{B}^2_{C_b^\theta(H;H)}\mathbb{E}\left[\left(\int^t_0 (t-s)^{-\beta}\norm{X_{1,n}(s)-X_{2,n}(s)}_Hds\right)^2\right]\notag\\
&\leq K^2_TM^2_T\norm{B}^2_{C_b^\theta(H;H)}\mathbb E\left[\left(\int_0^t(t-s)^{-\beta/2}(t-s)^{-\beta/2}\|X_{1,n}(s)-X_{2,n}(s)\|_Hds\right)^2\right]\notag\\
&\leq K^2_TM^2_T\norm{B}^2_{C_b^\theta(H;H)}\left(\int_0^t(t-s)^{-\beta}ds\right)\int_0^t(t-s)^{-\beta}\mathbb E\left[\|X_{1,n}(s)-X_{2,n}(s)\|_H^2\right]ds\notag\\
&\leq (1-\beta)^{-2}K^2_TM^2_T\norm{B}^2_{C_b^\theta(H;H)}T^{2(1-\beta)}\sup_{t\in[0,T]}\Delta_n(t).
\end{align}

As far as  \eqref{I4} and \eqref{I6} are concerned, from \eqref{P5}, \eqref{stimaopn}, \eqref{holdrBn} and \eqref{stima-n}  we get
\begin{align}
I_4+I_5&\leq T^2\sup_{t\in[0,T]}\norm{e^{tA}}^2_{\mathcal{L}(H)}\|B\|^2_{C^\theta_b(H;H)}\|\widetilde{\mathscr L}\|_{\mathcal{L}(H)}^2\Pi_{n},\label{SI45}\\
I_6+I_7&\leq T^2M_T^2\sup_{t\in[0,T]}\norm{e^{tA}}^2_{\mathcal{L}(H)}\|B\|^4_{C^\theta_b(H;H)}\|\widetilde{\mathscr L}\|_{\mathcal{L}(H)}^2\Pi_{n},\label{SI67}
\end{align}
where 
\begin{align*}
\Pi_{n}&:=\sup_{t\in [0,T]}\mathbb{E}\left[\norm{X_{1,n}(t)-X_1(t)}_H^{2\theta}\right]+\sup_{t\in [0,T]}\mathbb{E}\left[\norm{X_{2,n}(t)-X_2(t)}_H^{2\theta}\right].
\end{align*}
It remains to deal with \eqref{I8}. By applying the It\^o isometry, we get
\begin{align}\label{itooo}
I_8 & = \int_0^t\mathbb{E}\left[\norm{e^{(t-s)A_n}\widetilde{\mathscr L}_n\left[DU_n(s,X_{1,n}(s))-DU_n(s,X_{2,n}(s))\right]G_n}_{\mathscr L_2(U;H)}^2\right]dsdt.
\end{align}

Let us fix an orthonormal basis $\{u_\ell:\ell\in\N\}$ of $U$. We set 
\[
T_n=T_n(s,t):=e^{(t-s)A_n}\widetilde{\mathscr L}_n\left[DU_n(s,X_{1,n}(s))-DU_n(s,X_{2,n}(s))\right]G_n
\]
for every $s,t\in[0,T]$ with $s\leq t$. By Hypotheses \ref{hyp:goal-addo}(ii) we get 
\begin{align*}
&\int_0^t\mathbb{E}\left[\norm{T_n}_{L_2(U;H)}^2\right]ds  \leq \int_0^t\mathbb{E}\left[\sum_{\ell=1}^\infty\norm{T_n u_\ell}_H^2\right]ds
\\= & \int_0^T\mathbb{E}\left[\sum_{\ell=1}^\infty\sum_{k=1}^n\left(\sum_{h=1}^{d_k}\langle T_n u_\ell,e^k_h\rangle_H^2+\sum^{d_k}_{i,j=1,\;i\neq j}\langle T_n u_\ell,e^k_i\rangle_H\langle T_n u_\ell,e^k_j\rangle_H\scal{e^k_i}{e^k_j}_H\right)\right]ds \\
\\ \leq & \int_0^t\mathbb{E}\left[\sum_{\ell=1}^\infty\sum_{k=1}^n\left(\sum_{h=1}^{d_k}\langle T_n u_\ell,e^k_h\rangle_H^2+\frac{1}{2}\sum^{d_k}_{i,j=1,\;i\neq j}\left(\langle T_n u_\ell,e^k_i\rangle_H^2+\langle T_n u_\ell,e^k_j\rangle_H^2\right)\right)\right]ds \\
\\ = & \int_0^t\mathbb{E}\left[\sum_{\ell=1}^\infty\sum_{k=1}^n\left(\sum_{h=1}^{d_k}\langle T_n u_\ell,e^k_h\rangle_H^2+\frac{1}{2}(d_k-1)\left(\sum^{d_k}_{i=1}\langle T_n u_\ell,e^k_i\rangle_H^2+\sum^{d_k}_{j=1}\langle T_n u_\ell,e^k_j\rangle_H^2\right)\right)\right]ds \\
\\ = & \int_0^t\mathbb{E}\left[\sum_{\ell=1}^\infty\sum_{k=1}^n d_k\sum_{i=1}^{d_k}\langle T_n u_\ell,e^k_i\rangle_H^2\right]ds\leq d\int_0^T\mathbb{E}\left[\sum_{\ell=1}^\infty\sum_{k=1}^n \sum_{i=1}^{d_k}\langle T_n u_\ell,e^k_i\rangle_H^2\right]ds \\
= & d\int_0^t\sum_{\ell=1}^\infty\sum_{k=1}^n\sum_{i=1}^{d_k}\mathbb{E}\left[\langle e^{(t-s)A_n}\widetilde{\mathscr L}_n\left[DU_n(s,X_{1,n}(s))-DU_n(s,X_{2,n}(s))\right]G_n u_\ell,e^k_i\rangle_H^2\right]ds \\
\leq & d \int_0^t\sum_{\ell=1}^\infty\sum_{k=1}^n\sum_{i=1}^{d_k}e^{2(t-s){\rm Re}(\rho^k_{i})} d^2\zeta_k^2\mathbb{E}\left[\langle \left[DU_n(s,X_{1,n}(s))-DU_n(s,X_{2,n}(s))\right]G_nu_\ell,e^k_i\rangle_H^2\right]ds,
\end{align*}
where we have used the fact that for every $k=1,\ldots,n$ we have
\begin{align*}
\sum_{i=1}^{d_k}\langle \widetilde{\mathscr L}_nF,e_i^k\rangle_H^2
= & \sum_{i=1}^{d_k}\langle F,\widetilde{\mathscr L}_n e_i^k\rangle_H^2
= \sum_{i=1}^{d_k}\left(\sum_{j=1}^{d_k}\zeta_{k,i,j}\langle F, e_j^k\rangle_H\right)^2 \\
\leq & d_k \sum_{i=1}^{d_k}\sum_{j=1}^{d_k}|\zeta_{k,i,j}|^2\langle F, e_j^k\rangle_H^2\leq d^2\zeta_k^2\sum_{j=1}^{d_k}\langle F,e_j^k\rangle_H^2
\end{align*}
with $\zeta_{k,I,j}:=\scal{\widetilde{\mathscr{L}}^*e^k_i}{e^k_j}$, see Remark \ref{rmk:cont_dec_lambda_n}(i).

Setting $U_{n,k,i}:=\scal{U_n}{e^k_i}_H$ and $B^k_i:=\scal{B}{e^k_i}_H$ by \eqref{Gtilda}, \eqref{stima-n-scalare} and \eqref{scambio}, we obtain 
\begin{align*}
& \int_0^t\mathbb{E}\left[\norm{T_n}_{L_2(U;H)}^2\right]ds \\
\leq &   d^3\int_0^t\sum_{\ell=1}^\infty\sum_{k=1}^n|\zeta_k|^2\sum_{i=1}^{d_k}e^{2(t-s){\rm Re}(\rho^k_{i})}\mathbb{E}\left[\langle \nabla U_{n,k,i}(s,X_{1,n}(s))-\nabla U_{n,k,i}(s,X_{2,n}(s)),\widetilde{G}_n\mathcal{V} u_\ell\rangle_H^2\right]ds \\
= &   d^3\int_0^t\sum_{\ell=1}^\infty\sum_{k=1}^n|\zeta_k|^2\sum_{i=1}^{d_k}e^{2(t-s){\rm Re}(\rho^k_{i})}\mathbb{E}\left[\langle \widetilde{G}_n^*(\nabla U_{n,k,i}(s,X_{1,n}(s))-\nabla U_{n,k,i}(s,X_{2,n}(s))),\mathcal V u_\ell\rangle_H^2\right]ds \\
\leq & d^3\|\mathcal{V}^*\|_{\mathcal L(U)}^2 \int_0^t\sum_{k=1}^n|\zeta_k|^2\sum_{i=1}^{d_k}e^{2(t-s){\rm Re}(\rho^k_{i})}\mathbb{E}\left[\|\widetilde{G}_n^*(\nabla U_{n,k,i}(s,X_{1,n}(s))-\nabla U_{n,k,i}(s,X_{2,n}(s)))\|_H^2\right]ds \\
\leq & d^3M_{T}^2 \|\mathcal{V}^*\|_{\mathcal L(U)}^2 \int_0^t\sum_{k=1}^n|\zeta_k|^2\sum_{i=1}^{d_k}e^{2(t-s){\rm Re}(\rho^k_{i})}\|B^k_i\|_{C_b^\theta(H_n)}^2\mathbb{E}\left[\|X_{1,n}(s)-X_{2,n}(s)\|_H^2\right]ds.
\end{align*}
Recalling that ${\rm Re}(\rho^k_{i})<0$ for every $k\in\N$ and every $i\in\{1,\ldots,d_k\}$, we obtain
\begin{align*}
I_8
\leq & d^3M_{T}^2\|\mathcal{V}^*\|_{\mathcal L(U)}^2\sup_{t\in [0,T]}\Delta_n(t)\sum_{k=1}^n|\zeta_k|^2\sum_{i=1}^{d_k}\int_0^Te^{2(t-s){\rm Re}(\rho^k_{i})}\|B^k_i\|_{C_b^\theta(H_n)}^2dsdt\notag \\
= & dM_{T}^2\|\mathcal{V}^*\|_{\mathcal L(U)}^2\sup_{t\in [0,T]}\Delta_n(t)\sum_{k=1}^n|\zeta_k|^2\sum_{i=1}^{d_k}\frac{\|B^k_i\|_{C_b^\theta(H_n)}^2}{-2{\rm Re}(\rho_i^k)}.
\end{align*}
By \eqref{conv_serie_holder}, we infer that there exists a positive constant $\overline C$, independent of $n\in\N$, such that
\begin{align}\label{SI8-addo}
I_8
\leq dM_{T}^2 \|\mathcal{V}^*\|_{\mathcal L(U)}^2\overline C\sup_{t\in [0,T]}\Delta_n(t).   
\end{align}
By \eqref{SI1}, \eqref{SI2}, \eqref{SI3}, \eqref{SI45}, \eqref{SI67} and \eqref{SI8-addo}, there exists a positive constant $K_T>0$ such that 
\begin{align}
\Delta_n(t)
\leq & K_T\Big[\norm{x_1-x_2}^2_H+M_{T}^2\sup_{t\in[0,T]}\Delta_n(t)+(1+M_T^2)\Pi_n \Big] 
\label{SS3}
\end{align}
for every $t\in[0,T]$. By \eqref{supE}, it follows that $\Delta_n\to \Delta$ as $n$ goes to $\infty$. Moreover, since $\theta<1$ then $2/2\theta>1$, so by the Jensen's inequality and \eqref{supE} we deduce that $\Pi_n\to0$ as $n\to\infty$. Letting $n\to \infty$ in \eqref{SS3}, we deduce that
\begin{equation}\label{SS4}
\Delta(t)\leq K_T\Big[\norm{x_1-x_2}^2_H+M_{T}^2\sup_{t\in[0,T]}\Delta(t)\Big], \qquad \forall t\in[0,T].
\end{equation}
Noticing that, from definition, $M_T\rightarrow 0$ as $T\rightarrow 0$, we deduce that if $T>0$ is small enough, then
\begin{equation*}
\sup_{t\in[0,T]}\Delta(t)\leq 2K_{T}\norm{x_1-x_2}^2_H.
\end{equation*}
for some positive constant $C_{T}$. The statement for general $T>0$ follows by standard arguments.
\end{proof}

Now we show that it is possible to drop Hypothesis \ref{hyp:goal-addo}(ii) under the assumption that $G\in\mathcal L_2(U;H)$. In particular, we will establish that for every $T > 0$ and $x \in H$, pathwise uniqueness holds true for equation \eqref{SDE}, and the unique mild solution is Lipschitz continuous with respect to a norm stronger than the norm considered in \eqref{lip-scarsa}.

\begin{thm}\label{pathwiseuniqueneSS}
Assume that Hypotheses \ref{hyp:finito-dimensionale} and \ref{hyp:goal-addo}(i) are satisfied, and assume that $\mathcal{V}\in\mathcal L_2(U;U)$. Then, for every $T>0$ and $x\in H$ pathwise uniqueness holds true for equation \eqref{SDE}. Moreover, for every $T>0$ there exists $C_T>0$ such that for every $x_1,x_2\in H$, if $X(\cdot,x_1)$, $X(\cdot,x_2)$ are two weak mild solutions defined on the same filtered probability space with respect to the same cylindrical Wiener process $W$, then we have
\begin{equation}\label{lip-forte}
\mathbb{E}\left[\sup_{t\in [0,T]}\|X(t,x_1)-X(t,x_2)\|^2_H\right]\leq C_T\norm{x_1-x_2}^2_H.
\end{equation}
\end{thm}
\begin{proof}
The proof is a slight modification of the one of Theorem \ref{pathwiseuniqueness}. For $i=1,2$ we denote by $X_i$ the process $X(\cdot,x_i)$ and, for every $x_1,x_2\in H$ and every $n\in\N$, we set
\[
\Delta:=\mathbb{E}\left[\sup_{t\in [0,T]}\norm{X_{1}(t)-X_{2}(t)}^2_H\right],\quad \Delta_n:=\mathbb{E}\left[\sup_{t\in [0,T]}\norm{X_{1,n}(t)-X_{2,n}(t)}^2_H\right].
\]
The estimates which change with respect to those in the proof of Theorem \ref{pathwiseuniqueness} are the estimate of the terms corresponding to $I_3$ and $I_8$. As far as the estimate of the term corresponding to $I_3$ is considered, from Hypothesis \ref{hyp:goal-addo}(i) we get
\begin{align*}
&\mathbb{E}\left[\sup_{t\in[0,T]}\norm{A_n\int^t_0e^{(t-s)A_n}\widetilde{\mathscr L}_n\left(U_{n}(s,X_{1,n}(s))-U_{n}(s,X_{2,n}(s))\right)ds}^2_H\right] \\
\leq & (1-\beta)K_T^2M_T^2\|B\|_{C^\beta_b(H;H)}^2T^{1-\beta}\mathbb E\bigg[\sup_{t\in[0,T]}\int_0^t(t-s)^{-\beta}\|X_{1,n}(s)-X_{2,n}(s)\|^2_Hds\bigg] \\
\leq & (1-\beta)^2K_T^2M_T^2\|B\|_{C^\beta_b(H;H)}^2T^{2(1-\beta)}\mathbb E[\Delta_n].
\end{align*}
The estimate of the term corresponding to $I_8$ (see \eqref{I8}) is more direct. Indeed, in this setting by the Kotelentz inequality, \eqref{Gtilda}, \eqref{stimaopn}, \eqref{holdrBn} and \eqref{stima-n} we get 
\begin{align*}
   I_8&:=\mathbb{E}\left[\sup_{t\in [0,T]}\norm{\int_0^te^{(t-s)A_n}\widetilde{\mathscr L}_n(DU_n(s,X_{1,n}(s))-DU_n(s,X_{2,n}(s)))\widetilde{G}_n\mathcal{V}dW(s)}^2_H\right]\\
  &\leq C\int_0^T\mathbb{E}\left[\norm{\widetilde{\mathscr L}_n(DU_n(s,X_{1,n}(s))-DU_n(t,X_{2,n}(s))))\widetilde{G}_n\mathcal{V}}^2_{\mathcal{L}_2(U;H)}\right]ds\\
   &\leq C M_T^2\|B\|^2_{C^\theta_b(H;H)}T\Delta_n\|\widetilde{\mathscr L}\|_{\mathcal{L}(H)}^2\|\mathcal{V}\|^2_{\mathcal{L}_2(U;U)}.
\end{align*}
By similar calculations to those in the proof of Theorem \ref{pathwiseuniqueness}, we obtain
\[
\Delta_n\leq K_T\left[\norm{x_1-x_2}^2_H+M_{T}^2\Delta_n+(1+M_T^2)\Pi_n\right],
\]
where $K_T$ is a positive constant and 
\begin{align*}
\Pi_{n}&:=\mathbb{E}\left[\sup_{t\in [0,T]}\norm{X_{1,n}(t)-X_1(t)}_H^{2\theta}+\sup_{t\in [0,T]}\norm{X_{2,n}(t)-X_2(t)}_H^{2\theta}\right].
\end{align*}
Letting $n$ tend to infinity and exploiting \eqref{Esup}, it follows that
\begin{align*}
&
\Delta \leq K_T\norm{x-y}^2_H
+K_TM_{T}^2\Delta
\end{align*}
and so the statement follows recalling that $M_T$ vanishes as $t$ tends to $0$.
\end{proof}

\begin{coro}\label{Strong}
Assume Hypotheses that \ref{hyp:finito-dimensionale} and \ref{hyp:goal-addo} hold and that $(e^{tA})_{t\geq 0}$ is a compact semigroup. Then strong existence holds true for equation \eqref{SDE}.    
\end{coro}
\begin{proof}
By \cite[Proposition 3]{Cho-Gol1995}, it follows that there exists a weak mild solution to \eqref{SDE}. Further, Theorem \eqref{pathwiseuniqueness} gives pathwise uniqueness for \eqref{SDE}. Therefore, from \cite{On04}, which states that weak existence and pathwise uniqueness for equation \eqref{SDE} imply strong existence, we obtain the desired result.
\end{proof}

\begin{rmk}
Let $T>0$. Assume that the assumptions of Theorem \ref{pathwiseuniqueness} hold true. Hence, for every $x\in H$ the SPDE \eqref{SDE} has a pathwise unique mild solution $\{X(t,x)\}_{t\in [0,T]}$. For every $n\in\N$ and $x\in H$ we consider the unique mild solution $\{\widehat{X}_n(t,x)\}_{t\in [0,T]}$ to 
\begin{align*}
\left\{
\begin{array}{ll}
\displaystyle  d\widehat{X}_n(t)=A_n\widehat{X}_n(t)dt+\widetilde{\mathscr{L}}_n B_n(\widehat{X}_n(t))dt+G_ndW(t), \quad t\in[0,T],  \vspace{1mm} \\
\widehat{X_n}(0)= P_nx.
\end{array}
\right.
\end{align*}
By similar calculations to the one in \cite[Section 4.4]{AddBig}, for every $x\in H$ we get
\begin{equation*}
\lim_{n \to \infty}\sup_{t\in [0,T]}\mathbb{E}\left[\|\widehat{X}_{n}(t,x)-X(t,x)\|_H^2\right]=0.
\end{equation*}
Instead, if assumptions of Theorem \ref{pathwiseuniqueneSS} hold true (so $G\in\mathcal{L}_2(U;H)$) then for every $x\in H$ we get
\begin{equation*}
\lim_{n \to \infty}  \mathbb{E}\left[\sup_{t\in [0,T]}\|\widehat{X}_{n}(t,x)-X(t,x)\|_H^2\right]=0.
\end{equation*}
\end{rmk}

\section{Examples}\label{Examples}

\subsection{Heat equation with structure condition}\label{Heatcase}
Consider the SPDE
\begin{align}\label{eqFObeta}
\left\{
\begin{array}{ll}
\displaystyle  dX(t)=-(-\Delta)^{\beta}X(t)dt+(-\Delta)^{-\sigma}B(X(t))dt+(-\Delta)^{-\gamma/2}dW(t), \quad t\in[0,T],  \vspace{1mm} \\
X(0)=x\in H,
\end{array}
\right.
\end{align}
where $B\in C^\theta_b(H;H)$, for some $\theta\in (0,1)$, $\beta,\gamma,\sigma\geq 0$ and $\Delta$ is the realization of the Laplace operator with periodic boundary conditions in $H=L^2([0,2\pi]^m)$ with $m=1,2,3$. We are going to show that the pathwise uniqueness holds true for SPDE \eqref{eqFObeta}.

We recall that $A$ is self-adjoint and there exists an orthonormal basis $\{e_k:k\in\N\}$ of $H$ consisting in eigenvectors of $A$. Hence, the spaces $\{H_n\}_{n\in\N}$ are given by
\[
H_0=\{0\},\quad H_n:={\rm span}\{e_1,...,e_n\},\quad n\in\N,
\]
verify Hypotheses \ref{hyp:finito-dimensionale}(iii) and Hypotheses \ref{hyp:goal-addo}(ii)(a)-(c) with $d_n=1$ for every $n\in\N$. Moreover, for every $k\in\N$, we have
\begin{equation}\label{autovalori}
\Delta e_k=-\lambda_ke_k, \quad \lambda_k\sim k^{2/m}.
\end{equation}
By easy computations, for every $t>0$ and $n\in\N$ we obtain that
\begin{align*}
& Q_t:=\int_0^te^{-2s(-\Delta)^{\beta}}(-\Delta)^{-\gamma}ds=\frac12(-\Delta)^{-(\beta+\gamma)}(\Id_H-e^{-2t(-\Delta)^{\beta}}),
\qquad Q_{t,n}=P_nQ_{t}P_n,
\end{align*}
where $Q_{t,n}$ is defined in Hypothesis \ref{hyp:finito-dimensionale}(iv) and $P_n$ is the orthogonal projection on $H_n$.

\begin{prop}
If $(m-2\beta)/2<\gamma< \beta\theta/(2-\theta)$ and $\sigma>\max\{0,(m-2\beta)/4\}$ then Hypotheses \ref{hyp:finito-dimensionale} and \ref{hyp:goal-addo} hold true. So for every $T>0$ and $x\in H$ the SPDE \eqref{eqFObeta} has unique mild solution $\{X(t,x)\}_{t\in [0,T]}$, and there exists $C_T>0$ such that for every $x,y\in H$ we have
\[
\sup_{t\in [0,T]}\mathbb{E}\left[\norm{X(t,x)-X(t,y)}_H\right]\leq C_T\norm{x-y}_H. 
\]
\end{prop}
\begin{proof}
To prove the statement it is enough to show that the assumptions of Theorem \ref{pathwiseuniqueness} are verified. In view of \cite[Proposition 5.13]{AddBig}, if $(m-2\beta)/2<\gamma< \beta\theta/(2-\theta)$ then Hypotheses \ref{hyp:finito-dimensionale} hold true. By \cite[Proposition 2.1.1]{Lun95} (with $A=-(-\Delta)^{\beta}$) Hypotheses \ref{hyp:goal-addo}(i) hold true. We only need to prove \eqref{conv_serie_holder_var}. If $\sigma>\max\{0,(m-2\beta)/4\}$ then  \eqref{conv_serie_holder_var} holds true since by \eqref{autovalori} we have
\[
\sum_{n\in\N}\frac{1}{\lambda_k^{2\sigma+\beta}}\sim \sum_{n\in\N}\frac{1}{k^{(4\sigma+2\beta)/m}}<\infty.
\]
Hence all the assumptions of Theorem \ref{pathwiseuniqueness} hold true.
\end{proof}

\subsection{Stochastic damped equations}
\label{sub:stoch_damp_eq}
We deal with a semilinear stochastic differential equation of the form
\begin{align}\label{damped_stoc_equation}
\left\{
\begin{array}{ll}
\displaystyle \frac{\partial^2y}{\partial t^2}(t)
= -\Lambda y(t)-\rho\Lambda^{\alpha}\left(\frac{\partial y}{\partial t}(t)\right)+\Lambda^{-\sigma}C\left(y(t),\frac{\partial y}{\partial t}(t)\right)+\Lambda^{-\gamma}d{W}(t), & t\in[0,T], \vspace{1mm} \\
y(0)=y_0, & \vspace{1mm} \\
\displaystyle \frac{\partial y}{\partial t}(0)=y_1,
\end{array}
\right.
\end{align}
with $\alpha\in(0,1)$ and $\rho,\sigma,\gamma\geq0$. We consider a self-adjoint operator $\Lambda:D(\Lambda)\subseteq U\to U$ of positive type with compact resolvent, simple eigenvalues $(\mu_n)_{n\in\N}$, blowing up as $n$ tends to infinity, and corresponding (non normalized) eigenvectors $\{e_n:n\in\N\}$. 

We set $H:=U\times U$. The operators  $A:D(A)\subseteq H \to H$, $G,\widetilde G:U\to H$ and $\widetilde{\mathscr L}:H\to H$ are respectively defined as
\begin{align}
\label{damped_def_op_A_G_Lambda}
& D(A):= \left\{\begin{pmatrix}
h_1 \\ h_2    
\end{pmatrix}:h_2\in D(\Lambda^{\frac12}), \ h_1+\rho\Lambda^{\alpha-\frac12}h_2\in D(\Lambda^{\frac12})\right\}, \notag \\ 
& A:=\begin{pmatrix}
0 & \Lambda^{\frac12} \\
-\Lambda^{\frac12} & -\rho \Lambda^{\alpha}
\end{pmatrix}, 
\quad \widetilde G=\begin{pmatrix}
0 \\ {\rm Id}   
\end{pmatrix}, \quad \widetilde{\mathscr L}=\begin{pmatrix}
0 & 0 \\
0 & \Lambda^{-\sigma}
\end{pmatrix}, \quad \mathcal V=\Lambda^{-\gamma},   \quad G=\widetilde G\Lambda^{-\gamma}  
\end{align}
and $B(h)=\widetilde G\widetilde B(h)$ for every $h=\begin{pmatrix}
h_1 \\ h_2    
\end{pmatrix}\in H$, where $\widetilde B(h)=C(\Lambda^{-\frac12}h_1,h_2)$ for every $h=\begin{pmatrix}
h_1 \\ h_2 \end{pmatrix}\in H$. It follows that, if $C\in C_b^\theta(H;U)$ for some $\theta\in(0,1)$, then $\widetilde B\in C_b^\theta(H;U)$ and $\|\widetilde B\|_{C_b^\theta(H;U)}\leq c\|C\|_{C_b^\theta(H;U)}$ for some positive constant $c$. \\
In this setting, equation \eqref{damped_stoc_equation} has the form
\begin{align*}
dX(t)=AX(t)dt+\widetilde{\mathscr L}B(X(t))dt+GdW(t), \quad t\in[0,T], \quad X(0)=x\in H,    
\end{align*}
where $x=\begin{pmatrix}
\Lambda^{\frac12}y_0 \\ y_1    
\end{pmatrix}\in H$ with $y_0\in D(\Lambda^{\frac12})$ and $y_1\in U$, $X=\begin{pmatrix}
X_1 \\ X_2    
\end{pmatrix}$, $y=\Lambda^{-\frac12}X_1$ and $\frac{\partial y}{\partial t}=X_2$.

If $\rho^2\neq 4\mu_n^{1-2\alpha}$ for every $n\in\N$, it follows that $A$ has simple eigenvalues $(\lambda_n^{\pm})_{n\in\N}$, given by \begin{align*}
\lambda_n^{\pm}
= \frac{-\rho\mu_n^{\alpha}\pm\sqrt{\rho^2\mu_n^{2\alpha}-4\mu_n}}{2}, \quad \lambda_n^++\lambda_n^-=-\rho \mu_n^\alpha, \quad \lambda_n^-\cdot \lambda_n^+=\mu_n \qquad n\in\N,
\end{align*}
with corresponding eigenvectors
\begin{align*}
\Phi_n^+=\begin{pmatrix}
\mu_n^{\frac12}e_n \\ \lambda_n^+ e_n    
\end{pmatrix},
\qquad \Phi_n^-=\chi_n\begin{pmatrix}
\mu_n^{\frac12}e_n \\ \lambda_n^- e_n    
\end{pmatrix}, \qquad n\in\N.
\end{align*}
We also notice that
\begin{align*}
A^*=\begin{pmatrix}
0 & -\Lambda^{\frac12} \\
\Lambda^{\frac12} & -\rho \Lambda^\alpha
\end{pmatrix},    
\end{align*}
with same eigenvalues $(\lambda_n^{\pm})_{n\in\N}$ of $A$ and corresponding eigenvectors
\begin{align*}
\Psi_n^+=\begin{pmatrix}
-\mu_n^{\frac12}e_n \\ \lambda_n^+ e_n    
\end{pmatrix},
\qquad \Psi_n^-=\chi_n\begin{pmatrix}
-\mu_n^{\frac12}e_n \\ \lambda_n^- e_n    
\end{pmatrix}, \qquad n\in\N.    
\end{align*}
We decompose $H=H^++H^-$ (non-orthogonal, direct sum), where
\begin{align*}
{H^+}=\overline{\{\Phi_n^+:n\in\N\}}=\overline{\{\Psi_n^+:n\in\N\}}, \qquad {H^-}=\overline{\{\Phi_n^-:n\in\N\}}=\overline{\{\Psi_n^-:n\in\N\}}.    
\end{align*}
Every $h\in H$ can be uniquely written as $h=h^++h^-$, where $h^+\in H^+$ and $h^-\in H^-$. Further, we can read the action of $A, e^{tA}$ and $R(\lambda,A)$ in terms of such a decomposition:
\begin{align}
& Ah=\sum_{n=1}^\infty\left(\lambda_n^+\langle h^+,\Phi_n^+\rangle \Phi_n^++\lambda_n^-\langle h^-,\Phi_n^-\rangle \Phi_n^-\right), \qquad h\in D(A), \label{damped_az_A_H+H-}\\
& e^{tA} = \sum_{n=1}^\infty\left(e^{\lambda_n^+t}\langle h^+,\Phi_n^+\rangle \Phi_n^++e^{\lambda_n^-t}\langle h^-,\Phi_n^-\rangle \Phi_n^-\right), \qquad h\in H, \label{damped_az_etA_H+H-}\\ 
& R(\lambda,A) = \sum_{n=1}^\infty\left(\frac{1}{\lambda-\lambda_n^+}\langle h^+,\Phi_n^+\rangle \Phi_n^++\frac{1}{\lambda-\lambda_n^-}\langle h^-,\Phi_n^-\rangle \Phi_n^-\right), \qquad h\in H \notag
\end{align}
for every $t\geq0$ and every $\lambda\in \rho(A)$. We also recall that 
\begin{align}
\label{damped_op_G_dec}
\widetilde Gu=\begin{pmatrix}
0 \\ u    
\end{pmatrix}
= \sum_{n=1}^\infty\left(b_n^+u_n\Phi_n^++b_n^-u_n\Phi_n^-\right),   \qquad u\in U, 
\end{align}
where $u_n=\langle u,e_n/\|e_n\|_U\rangle_U$ for every $u\in U$ and every $n\in\N$. Since 
\begin{align}
\label{damped_op_lambda}
\widetilde{\mathscr L} k=\begin{pmatrix}
0 & 0 \\
0 & \Lambda^{-\sigma}
\end{pmatrix}\begin{pmatrix}
k_1 \\ k_2    
\end{pmatrix}=\begin{pmatrix}
0 \\ \Lambda^{-\sigma}k_2
\end{pmatrix}, \qquad k=\begin{pmatrix}
k_1\\ k_2    
\end{pmatrix}\in H,     
\end{align}
from \eqref{damped_op_G_dec} and \eqref{damped_op_lambda} we deduce that
\begin{align}
\label{damped_op_lambda_dec}
\widetilde{\mathscr L} k
= \sum_{n=1}^\infty\left(b_n^+\mu_n^{-\sigma}(k_2)_n\Phi_n^++b_n^-\mu_n^{-\sigma}(k_2)_n\Phi_n^-\right)
= \sum_{n=1}^\infty\mu_n^{-\sigma}(k_2)_n\left(b_n^+\Phi_n^++b_n^-\Phi_n^-\right)
\end{align}
for every $k\in H$. Let us prove that Hypothesis \ref{hyp:goal-addo}(ii)(c) is fulfilled. Fix $n\in\N$. In this case $K_n:={\rm span}\{\Phi_n^+,\Phi_n^-\}$ and from \eqref{damped_op_lambda_dec} we infer that
\begin{align}
\widetilde{\mathscr L} \Phi_n^+
= & \mu_n^{-\sigma}\lambda_n^+\|e_n\|_U(b_n^+\Phi_n^++b_n^-\Phi_n^-) \in K_n, 
\label{damped_Lambda_tilde_K_n+}\\
\widetilde{\mathscr L} \Phi_n^-
= & \mu_n^{-\sigma}\chi_n\lambda_n^-\|e_n\|_U(b_n^+\Phi_n^++b_n^-\Phi_n^-) \in K_n,
\label{damped_Lambda_tilde_K_n-}
\end{align}
since for every $m\in\N$ we get $((\Phi_n^+)_2)_m=\delta_{mn}\lambda_n^+\|e_n\|_U$ and $((\Phi_n^-)_2)_m=\delta_{mn}\chi_n\lambda_n^-\|e_n\|_U$.

Finally, we recall that, from the above construction, it follows that, if $\alpha\in\left(0,\frac12\right]$, then \begin{align}
& {\rm Re}(\lambda_n^{+}), {\rm Re}(\lambda_n^{-})\sim -\mu_n^\alpha, \quad |\lambda_n^{+}|,|\lambda_n^{-}|\sim \mu_n^{\frac12}, \quad  \,  \|e_n\|_{U}\sim \mu_n^{-\frac12}, \quad |\lambda_n^+-\lambda_n^-|\sim \mu_n^{\frac12}, \notag \\
& b_n^{\pm}\sim {\rm const}(b), \quad \chi_n\sim {\rm const}(\chi),
\label{damped_stime_coefficienti_avl}
\end{align}
definitively with respect to $n\in\N$ (see also formulae \cite[(2.3.14)-(2.3.18)]{trig}). The case $\alpha\in\left[\frac12,1\right)$ is analogously treated. We stress that, in this case, the asymptotic behaviour in \eqref{damped_stime_coefficienti_avl} is replaced by
\begin{align}
& {\rm Re}(\lambda_n^+)\sim -\mu_n^{1-\alpha}, \quad {\rm Re}(\lambda_n^-)\sim -\mu_n^\alpha, \quad |\lambda_n^+|\sim \mu_n^{1-\alpha}, \quad |\lambda_n^-|\sim\mu_n^\alpha, \quad  \|e_n\|_{U}\sim \mu_n^{-\frac12}, \notag \\
& |\lambda^-_n-\lambda_n^+|\sim \mu_n^{\alpha}, \quad b_n^-\sim {\rm const}(b), \quad \chi_n,b_n^+\sim \mu_n^{\frac12-\alpha}
\label{damped_stime_coeff_avl_2}
\end{align}
definitively with respect to $n\in\N$. 

Thanks to \eqref{damped_stime_coefficienti_avl} and \eqref{damped_stime_coeff_avl_2}, we are able to estimate the asymptotic behaviour of $\zeta_n$, introduced in Hypothesis \ref{hyp:goal-addo}(ii)(c).
\begin{lemma}
For every $n\in\N$, we get $\zeta_n\sim \mu_n^{-\sigma}$ definitively, for every $\alpha\in(0,1)$.
\end{lemma}
\begin{proof}
Assume that $\alpha\in \left(0,\frac12\right]$. Then, from \eqref{damped_Lambda_tilde_K_n+}, \eqref{damped_Lambda_tilde_K_n-} and \eqref{damped_stime_coefficienti_avl} we infer that
\begin{align*}
\|\widetilde{\mathscr L} \Phi_n^+\|_{H}, \|\widetilde{\mathscr L} \Phi_n^-\|_{H}
\sim \mu_n^{-\sigma}.
\end{align*}
Since $\{\Phi_n^+,\Phi_n^-\}$ is a normalized basis of $K_n$, it follows that $\zeta_n\sim \mu_n^{-\sigma}$.

As far as $\alpha\in\left[\frac12,1\right)$ is concerned, we notice that estimates \eqref{damped_stime_coeff_avl_2}, combined with \eqref{damped_Lambda_tilde_K_n+} and \eqref{damped_Lambda_tilde_K_n-}, give
\begin{align*}
\|\widetilde{\mathscr L}\Phi_n^+\|_H\sim \mu_n^{-\sigma+\frac12 -\alpha}, \quad \|\widetilde{\mathscr L}\Phi_n^-\|_H\sim \mu_n^{-\sigma}.    
\end{align*}
Since $(\mu_n)_{n\in\N}$ positively diverges as $n$ goes to infinity and $\alpha\geq \frac12$, it follows that $\mu_n^{-\sigma+\frac12-\alpha}\leq \mu_n^{-\sigma}$ for every $n\in\N$.
\end{proof}

Now we prove the following crucial estimate.
\begin{prop}
\label{prop:stima_espl}
Let us fix $T>0$. For every $\alpha\in(0,1)$ there exist positive constants $C=C(\alpha,T)$ and $\beta$ such that
\begin{align*}
\|Ae^{tA}\widetilde{\mathscr L}\|_{\mathcal L(H)}\leq C t^{-\beta}, \qquad t\in(0,T].    
\end{align*}
In particular, $\beta=\frac{(1-2\sigma)\vee0}{2\alpha}$ if $\alpha\in\left(0,\frac12\right)$, and $\beta=\left(1-\frac{\sigma}{\alpha}\right)\vee0$ if $\alpha\in\left[\frac12,1\right)$. Moreover, for every $n\in\N$ we get
\begin{align*}
\|A_ne^{tA_n}\widetilde{\mathscr L}_n\|_{\mathcal L(H_n)}\leq C t^{-\beta}, \qquad t\in(0,T],        
\end{align*}
with the same constants $C$ and $\beta$.
\end{prop}
\begin{proof}
From \eqref{damped_az_A_H+H-} and \eqref{damped_az_etA_H+H-}, it follows that $e^{tA}h\in D(A)$ for every $t>0$ and every $h\in H$, and we deduce the explicit expression
\begin{align}
Ae^{tA}h=\sum_{n=1}^\infty\left(\lambda_n^+e^{\lambda_n^+t}\langle h^+,\Phi_n^+\rangle\Phi_n^++\lambda_n^-e^{\lambda_n^-t}\langle h^-,\Phi_n^-\rangle \Phi_n^-\right).    
\label{damped_espr_AetAh}
\end{align}
If we consider $h=\widetilde{\mathscr L} k$ for some $k\in H$, then from \eqref{damped_op_lambda_dec} and \eqref{damped_espr_AetAh} we infer that
\begin{align}
\label{damped_esprAetALambdak}
Ae^{tA}\widetilde{\mathscr L} k
= \sum_{n=1}^\infty\mu_n^{-\sigma}(k_2)_n\left(\lambda_n^+e^{\lambda_n^+t}b_n^+\Phi_n^++\lambda_n^-e^{\lambda_n^-t}b_n^-\Phi_n^-\right).
\end{align}
Let us estimate the $H$-norm of $Ae^{tA}\widetilde{\mathscr L} k$. To this aim, we separately consider the cases $\alpha\in\left(0,\frac12\right)$ and $\alpha\in\left[\frac12,1\right)$. In the following, $c$ is a positive constant which may vary line by line.

If $\alpha\in\left(0,\frac12\right)$, then, by taking advantage from \eqref{damped_stime_coefficienti_avl} and \eqref{damped_esprAetALambdak}, we deduce that
\begin{align*}
\|Ae^{tA}\widetilde{\mathscr L} k\|_H^2
\leq & 2{\rm const}(b)\left(\sum_{n=1}^\infty|\lambda_n^+|^2|e^{\lambda_n^+ t}|^2+\sum_{n=1}^\infty|\lambda_n^-|^2|e^{\lambda_n^- t}|^2 
\right)|\mu_n^{-\sigma}(k_2)_n|^2  \\
\leq & c\sum_{n=1}^\infty \mu_n^{1-2\sigma}e^{-2\rho\mu_n^\alpha t}|(k_2)_n|^2\leq ct^{-\frac{(1-2\sigma)\vee 0}{\alpha}}\|k\|_H^2,
\end{align*}
where the last inequality follows from the fact for every $a,b,d>0$, there exists $c>0$ such that
\begin{align*}
x^{a}e^{-bx^dt}\leq c t^{-\frac ad}, \qquad x\geq 1, \ t\in(0,T].
\end{align*}
This means that the $\mathcal L(H)$-norm of $Ae^{tA}\widetilde{\mathscr L}$ is summable at $t=0$ if $\sigma>\frac{1-2\alpha}{2}$.

If $\alpha\in\left[\frac12,1\right)$, then from \eqref{damped_stime_coeff_avl_2} and \eqref{damped_esprAetALambdak} we infer that
\begin{align*}
\|Ae^{tA}\widetilde{\mathscr L} k\|_H^2
\leq & 2\left(\sum_{n=1}^\infty|b_n^+\lambda_n^+|^2|e^{\lambda_n^+ t}|^2 +\sum_{n=1}^\infty|b_n^-\lambda_n^-|^2|e^{\lambda_n^- t}|^2 
\right) |\mu_n^{-\sigma}(k_2)_n|^2  \\
\leq & c\left(\sum_{n=1}^\infty \mu_n^{3-4\alpha-2\sigma}e^{-2\rho\mu_n^{1-\alpha} t}+\sum_{n=1}^{\infty}\mu_n^{2\alpha-2\sigma}e^{-2\rho\mu_n^{\alpha}t}\right)|(k_2)_n|^2 \\
\leq & c\left(t^{(-2+\frac{2\sigma+2\alpha-1}{1-\alpha})\wedge 0}+t^{(-2+\frac{2\sigma}{\alpha})\wedge 0}\right)\|k\|_H^2 \\
\leq & ct^{-2[(1-\frac{\sigma}{\alpha})\vee0]}\|k\|^2_H,
\end{align*}
since the time interval $[0,T]$ is bounded and $\alpha\in\left[\frac12,1\right)$, which give $t^{\frac{2\sigma+2\alpha-1}{1-\alpha}}\leq Ct^{\frac{2\sigma}{\alpha}}$ for some positive constant $C$ and every $t\in[0,T]$. Hence, the $\mathcal L(H)$-norm of $Ae^{tA}\widetilde{\mathscr L}$ is summable at $t=0$ if $\sigma>0$.
\end{proof}

\begin{rmk}
Now we explain why, in both the cases $\alpha\in\left(0,\frac12\right)$ and $\alpha\in\left[\frac12,1\right)$, the exponent $\beta$ has the expected value.

When $\alpha\geq \frac12$, this happens due to the fact that the operator $A$ is sectorial. Since $\|\mathscr L_n\|_{\mathcal L(H_n)}\sim \mu_n^{-\sigma}$ and $\|A_n\|_{\mathcal L(H_n)}\sim \mu_n^{\alpha}$, the action of $\mathscr L_n$ can be interpreted as the action of $A_n^{-\frac\sigma\alpha}$. As $n$ tends to infinity, one expect that the function $t\mapsto \|A_ne^{tA_n}A_n^{-\frac\sigma\alpha}\|_{\mathcal L(H_n)}$ behaves near $0$ as the function $t\mapsto \|Ae^{tA}A^{-\frac\sigma\alpha}\|_{\mathcal L(H_n)}$, which has a singularity of order $t^{-1+\frac\sigma\alpha}$ at $0$.

On the other hand, if $\alpha<\frac12$, then the operator $A$ generates a strongly continuous semigroup which is not analytic. Therefore, even if in this case $\|\mathscr L_n\|_{\mathcal L(H_n)}\sim \mu_n^{-\sigma}$ and $\|A_n\|_{\mathcal L(H_n)}\sim \mu_n^{\frac12}$, we do not expect that $\beta=1-2\sigma$, but rather a worse behavior, as indicated by the presence of the correction factor $2\alpha$ in the denominator, which is greater than $1$.
\end{rmk}

\subsubsection{The stochastic convolution}

The fact that the stochastic convolution
\begin{align*}
W_A(t):=\int_0^t e^{(t-s)A}GdW(s)    
\end{align*}
is well-defined for every $t\in[0,T]$, where $A$ and $G$ have been defined in \eqref{damped_def_op_A_G_Lambda}, follows from \cite[Proposition 5.3]{AddBig}, which proves that Hypothesis \ref{hyp:finito-dimensionale}(i) is verified.
\begin{prop}
\label{prop:damped_conv_stoc}
Let $A$ and $G$ be as in \eqref{damped_def_op_A_G_Lambda}. Assume that one of the following conditions holds true:
\begin{enumerate}[\rm(i)]
    \item $\Lambda^{-2\gamma}$ is a trace-class operator on $U$;
    \item there exist $\delta>0$ and a positive constant $c$ such that for every $n\in\N$ we have $\mu_n=c n^\delta$ and $\delta>\frac{1}{2\gamma+\alpha}$.
\end{enumerate}
Therefore, there exists $\eta\in(0,1)$ such that
\begin{align*}
\int_0^T t^{-\eta}{\rm Trace}_H\left[e^{tA}GG^*e^{tA^*}\right]dt<\infty.   
\end{align*}
\end{prop}
\subsubsection{The control problem}

For every $n\in\N$ and $t\in(0,T]$, we consider the control problem
\begin{align}
\left\{
\begin{array}{ll}
dY(\tau)=A_nY(\tau)d\tau+G_nu(\tau)d\tau,     & \tau\in(0,t], \\[1mm]
Y(0)=h\in H_n,
\end{array}
\right.
\label{damped_control_problem_n}
\end{align}
where $A_n=AP_n$, $G_n=P_nG$ and $u:[0,t]\to U$. We get
\begin{align*}
G_nu=\begin{pmatrix}
0 \\ \Lambda_n^{-\gamma}u
\end{pmatrix}
=\mu_n^{-\gamma}u_n\left(b_n^+\Phi_n^++b_n^-\Phi_n^-\right), \qquad u\in U,
\end{align*}
where $\Lambda_nu=\mu_nu_n$ with $u_n=\langle u,e_n/\|e_n\|_U\rangle_U$. From \cite{Zab08}, it follows that $e^{tA_n}(H_n)\subseteq Q_{t,n}(H_n)$ is equivalent to the null controllability of \eqref{damped_control_problem_n} and, if we set 
\begin{align*}
\mathcal E_{C,n}(t,h):=\left\{\|u\|_{L^2(0,t);U}:\eqref{damped_control_problem_n} \textrm{ admits a unique solution $Y$ with $Y(t)=0$}\right\},
\end{align*}
then $\mathcal E_{C,n}(t,h)=\Gamma_{t,n}h$. 

Here, we provide estimates similar to those of Theorems \cite[Theorems 5.4 \& 5.6]{AddBig}, for $\alpha\in\left(0,\frac12\right]$.  
\begin{thm}
\label{thm:stime_controllo_alpha<12}
Let $\alpha\in\left(0,\frac12\right]$ and $\gamma\geq0$. Then, system \eqref{damped_control_problem_n} is null-controllable. Further, for every $t>0$ there exists a positive constant $\overline c$, which depends on $\alpha$ and $\gamma$ but is independent of $n$ and $t$, if $t$ varies in a bounded interval, such that
\begin{align}
\label{stima_energia_damped_<1/2_compl}
\mathcal E_{C,n}(t,h)
\leq 
\begin{cases}
\displaystyle \frac{\overline c\|h\|_{H_n}}{t^{\frac12+\frac{\gamma}{\alpha}}}, & \gamma>\alpha, \\[5mm]
\displaystyle \frac{\overline c\|h\|_{H_n}}{t^{\frac32}}, & \gamma\leq \alpha,
\end{cases}
\qquad \forall h\in H_n.
\end{align}
Moreover, for every $\alpha\in\left(0,\frac12\right]$, for every $t>0$ there exists a positive constant $\overline c$, which depends on $\alpha$ and $\gamma$ but is independent of $n$ and $t$, if $t$ varies in a bounded interval, such that
\begin{align}
\label{stima_energia_damped_<1/2_direct_compl}
& \mathcal E_{C,n}(t,G_n a)
\leq \frac{\overline{c}\|G_na\|_{H_n}^2}{t^{\frac12+\frac{\gamma}{\alpha}}}, \qquad \forall a\in U, \ \alpha\in\left(0,\frac12\right].
\end{align}
\end{thm}
\begin{proof}
We follow the strategy in the proof of \cite[Theorem 5.4]{AddBig}. In particular, we fix $t\in(0,T]$, $h\in H_n$ and consider the control $u:[0,t]\to U$ defined as
\begin{align}
\label{control}
u(\tau)=
\begin{cases}
K_1 \psi_t(\tau)+K_2\psi_t'(\tau), & \forall \tau\in(0,t), \vspace{1mm} \\
0, & \tau=0, \ \tau=t,
\end{cases}
\end{align}
where $\psi_t(\tau)=-\Phi_t(\tau)e^{\tau A_n}h$ for every $\tau \in(0,t)$ and $\Phi_t:[0,t]\to \mathbb R$ is defined as $\Phi_t(\tau )=C_m \tau ^m(t-\tau )$ for every $\tau\in[0,t]$, $C_m$ is a normalizing constant which gives $\|\Phi_t\|_{L^1(0,t)}=1$ and $m\in\N$ satisfies $-2\gamma/\alpha+2m>-1$. Let us notice that $\psi_t$ is differentiable in $(0,t)$ and $\psi_t'(\tau )=-\Phi_t'(\tau )e^{\tau A_n}h-\Phi_t(\tau )A_ne^{\tau A_n}h$. The operators $K_1,K_2:H_n\to U_n$ are the row of the operator $K=[G_n|A_nG_n]$, with
\begin{align*}
[G_n|A_nG_n]=\begin{pmatrix}
0 & \Lambda_n^{\frac12-\gamma} \\
\Lambda_n^{-\gamma} & -\rho\Lambda_n^{\alpha-\gamma}
\end{pmatrix}
, \qquad 
K:=[G_n|A_nG_n]^{-1}=\begin{pmatrix}
\rho\Lambda_n^{\alpha-\frac12+\gamma} & \Lambda_n^{\gamma} \\
\Lambda_n^{-\frac12+\gamma} & 0
\end{pmatrix}.    
\end{align*}
In particular,
\begin{align*}
K_1k
= & K_1\begin{pmatrix}
    k_1 \\ k_2
\end{pmatrix}= \rho \Lambda_n^{\alpha-\frac12+\gamma} k_1+\Lambda_n^{\gamma}k_2\in U_n, \\   
K_2k
= & K_2\begin{pmatrix}
    k_1 \\ k_2
\end{pmatrix}=  \Lambda_n^{-\frac12+\gamma} k_1\in U_n
\end{align*}
for every $k=\begin{pmatrix}
 k_1 \\ k_2   
\end{pmatrix}\in H_n$. \\
Arguing as in Step ${1}$ in the proof of \cite[Theorem 5.4]{AddBig}, we infer that $u$ steers $h$ at $0$ at time $t$. It remains to estimate the $L^2(0,t;U)$-norm of $u$. To this aim, we recall that
$|\Phi_t(\tau)|\sim \frac{\tau^{m}}{t^{m+1}}$ and $|\Phi_t'(\tau)|\sim \frac{\tau^{m-1}}{t^{m+1}}$. Further, in the proof of \cite[Theorem 5.4]{AddBig} it has been proved that
\begin{align}
& \rho\Lambda_n^{\alpha-\frac12+\gamma}(e^{\tau A_n}h)_1+\Lambda_n^{\gamma}(e^{\tau A_n}h)_2 \notag \\
= & \sum_{k=1}^n[e^{\lambda_k^+\tau }\langle h^+,\Phi_k^+\rangle_H(\rho \mu_k^{\alpha-\frac12+\gamma}\mu_k^{\frac12}+\mu_k^\gamma\lambda_k^+)e_k
+e^{\lambda_k^-\tau }\langle h^-,\Phi_k^- \rangle_H \chi_k(\rho \mu_k^{\alpha-\frac12+\gamma}\mu_k^{\frac12}+\mu_k^\gamma\lambda_k^-)e_k] \notag \\
= & 
\sum_{k=1}^n\mu_k^\gamma[e^{\lambda_k^+\tau }\langle h^+,\Phi_k^+\rangle_H(\rho \mu_k^{\alpha}+\lambda_k^+)e_k
+e^{\lambda_k^-\tau }\langle h^-,\Phi_k^- \rangle_H \chi_k(\rho \mu_k^{\alpha}+\lambda_k^-)e_k] \notag \\
= & -\sum_{k=1}^n\mu_k^\gamma[e^{\lambda_k^+\tau }\langle h^+,\Phi_k^+\rangle_H\lambda_k^-e_k
+e^{\lambda_k^-\tau }\langle h^-,\Phi_k^- \rangle_H \chi_k\lambda_k^+e_k].
\label{forma_K_1psi_t}
\end{align}
and that
\begin{align}
K_2\psi_t'(\tau )
= & \Lambda_n^{-\frac12+\gamma}(\Phi'_t(\tau )(e^{\tau A_n}h)_1+\Phi_t(\tau )(A_ne^{\tau A_n}h)_1) \notag \\
= & \sum_{k=1}^n\mu_k^{-\frac12+\gamma}[\Phi'_t(\tau )(e^{\lambda_k^+\tau }\langle h^+,\Phi_k^+\rangle_H+e^{\lambda_k^-\tau }\langle h^-,\Phi_k^-\rangle_H\chi_k)\mu_k^{\frac12}e_k \notag \\
& +\Phi_t(\tau )(\lambda_k^+e^{\lambda_k^+\tau }\langle h^+,\Phi_k^+\rangle_H+\lambda_k^-e^{\lambda_k^-\tau }\langle h^-,\Phi_k^-\rangle_H\chi_k)\mu_k^{\frac12}e_k] \notag \\
= & \sum_{k=1}^n\mu_k^{\gamma}[\Phi'_t(\tau )(e^{\lambda_k^+\tau }\langle h^+,\Phi_k^+\rangle_H+e^{\lambda_k^-\tau }\langle h^-,\Phi_k^-\rangle_H\chi_k)e_k \notag \\
& +\Phi_t(\tau )(\lambda_k^+e^{\lambda_k^+\tau }\langle h^+,\Phi_k^+\rangle_H+\lambda_k^-e^{\lambda_k^-\tau }\langle h^-,\Phi_k^-\rangle_H\chi_k)e_k]
\label{forma_K_2psi_t'}
\end{align}
for every $\tau \in(0,t]$. 

{\bf Step ${\bm 1}$}. We prove \eqref{stima_energia_damped_<1/2_compl}. From \eqref{damped_stime_coefficienti_avl} and \eqref{forma_K_1psi_t} we infer that
\begin{align*}
\|K_1\psi_t(\tau)\|_U^2
\leq & |\Phi_t(\tau)|^2\sum_{k=1}^n\mu_k^{2\gamma}e^{-2\rho \mu_k^\alpha\tau}(\langle h^+,\Phi_k^+\rangle_H^2+\langle h^-,\Phi_k^-\rangle_H^2) \\
\leq & |\Phi_t(\tau)|^2\tau^{-\frac{2\gamma}{\alpha}}\sum_{k=1}^n(\langle h^+,\Phi_k^+\rangle_H^2+\langle h^-,\Phi_k^-\rangle_H^2), \qquad \tau\in(0,t).
\end{align*}
It follows that
\begin{align*}
\int_0^t\|K_1\psi_t(\tau)\|_U^2d\tau\leq & t^{-2m-2}\int_0^t\tau^{2m-\frac{2\gamma}{\alpha}}d\tau\sum_{k=1}^n    (\langle h^+,\Phi_k^+\rangle_H^2+\langle h^-,\Phi_k^-\rangle_H^2) 
\leq C t^{-1-\frac{2\gamma}{\alpha}}\|h\|_{H_n}^2.
\end{align*}
Analogously, from \eqref{damped_stime_coeff_avl_2} and \eqref{forma_K_2psi_t'} we deduce that
\begin{align*}
\|K_2\psi_t'(\tau)\|_U^2
\leq & C\sum_{k=1}^n(|\Phi_t'(\tau)|^2\mu_k^{2\gamma-1}e^{-2\rho\mu_k^\alpha\tau}+|\Phi_t(\tau)|^2\mu_k^{2\gamma}e^{-2\rho\mu_k^\alpha\tau})(\langle h^+,\Phi_k^+\rangle_H^2+\langle h^-,\Phi_k^-\rangle_H^2) \\
\leq & C(|\Phi_t'(\tau)|^2\tau^{(-\frac{2\gamma-1}{\alpha})\wedge0}+|\Phi_t(\tau)|^2\tau^{-\frac{2\gamma}{\alpha}})\sum_{k=1}^n(\langle h^+,\Phi_k^+\rangle_H^2+\langle h^-,\Phi_k^-\rangle_H^2)
\end{align*}
for every $\tau\in(0,t]$. We notice that
\begin{align*}
|\Phi_t'(\tau)|^2\tau^{(-\frac{2\gamma-1}{\alpha})\wedge 0}
\leq \frac{\tau^{2m-2+(-\frac{2\gamma-1}{\alpha})\wedge 0}}{t^{2m+2}}, \quad |\Phi_t(\tau)|^2\tau^{-\frac{2\gamma}{\alpha}}
\leq \frac{\tau^{2m-\frac{2\gamma}{\alpha}}}{t^{2m+2}}, \quad \tau\in(0,t).    
\end{align*}
Since $\frac1\alpha>2$, it follows that
\begin{align*}
\|K_2\psi_t'(\tau)\|_U^2\leq 
Ct^{-2m-2}
\begin{cases}
\displaystyle \tau^{2m-2} \|h\|_{H_n}^2, & \gamma\leq \alpha, \\[1mm]
\displaystyle \tau^{2m-\frac{2\gamma}{\alpha}}  \|h\|_{H_n}^2, & \gamma >\alpha
\end{cases}
\qquad \tau\in(0,t). 
\end{align*}
Therefore, the integrability of $\tau^{2m-\frac{2\gamma}{\alpha}}$ near $0$ gives
\begin{align*}
\int_0^t\|K_2\psi_t'(\tau)\|_U^2d\tau\leq & C t^{-2m-2}\int_0^t\tau^{2m-2}d\tau\|h\|_{H_n}^2
\leq Ct^{-3}\|h\|_{H_n}^2, \qquad \gamma\leq \alpha,
\end{align*}
and
\begin{align*}
\int_0^t\|K_2\psi_t'(\tau)\|_U^2d\tau\leq & C t^{-2m-2}\int_0^t\tau^{2m-\frac{2\gamma}{\alpha}}d\tau\|h\|_{H_n}^2
\leq Ct^{-1-\frac{2\gamma}{\alpha}}\|h\|_{H_n}^2, \qquad \gamma> \alpha.
\end{align*}
{\bf Step $\bm2$}. Now, we show that \eqref{stima_energia_damped_<1/2_direct_compl} hold true. To this aim, we consider the same control $u$ defined in \eqref{control} and $a\in U$. In the proof of \cite[Theorem 5.6]{AddBig} it has been shown that
for every $k=1,\ldots,n$,
\begin{align}
\label{exp_Ga^+Ga^-}
\langle (G_na)^+,\Phi_k^+\rangle_H
= & \frac{-a_k}{(\lambda_k^--\lambda_k^+)\|e_k\|_U},\qquad
\langle (G_na)^-,\Phi_k^-\rangle_H
=  \frac{-a_k}{\chi_k(\lambda_k^+-\lambda_k^-)\|e_k\|_U}.    
\end{align}
Replacing the above expressions in \eqref{forma_K_1psi_t}, we infer that
\begin{align*}
K_1\psi_t(\tau)
= \Phi_t(\tau)\sum_{k=1}^n\frac{\mu_k^{\gamma}a_k}{(\lambda_k^--\lambda_k^+)\|e_k\|_U}[e^{\lambda_k^+\tau}\lambda_k^--e^{\lambda_k^-\tau}\lambda_k^+]e_k, \qquad \tau\in(0,t),
\end{align*}
and so, from \eqref{damped_stime_coefficienti_avl}, it follows that
\begin{align*}
\|K_1\psi_t(\tau)\|_U^2
\leq & C |\Phi_t(\tau)|^2\sum_{k=1}^n\mu_k^{2\gamma}e^{-2\rho\mu_k^\alpha\tau}a_k^2
\leq C \tau^{-\frac{2\gamma}{\alpha}}|\Phi_t(\tau)|^2\sum_{k=1}^na_k^2, \qquad \tau\in(0,t).
\end{align*}
Hence, 
\begin{align*}
\int_0^t\|K_1\psi_t(\tau)\|_U^2d\tau\leq Ct^{-2m-2}\int_0^t\tau^{2m-\frac{2\gamma}{\alpha}}d\tau\|G_na\|_{H_n}^2
\leq Ct^{-1+(-\frac{2(\gamma-\beta)}{\alpha})\wedge 0}\|G_na\|_{H_n}^2.
\end{align*}
Let us compute the $\|\cdot\|_U$-norm of $K_2\psi_t'(\tau)$ for every $\tau\in(0,t)$. From \eqref{forma_K_2psi_t'} and \eqref{exp_Ga^+Ga^-} we infer that
\begin{align*}
K_2\psi_t'(\tau)
= & -\sum_{k=1}^n\frac{\mu_k^{\gamma}}{(\lambda_k^--\lambda_k^+)\|e_k\|_U}[\Phi_t'(\tau)(e^{\lambda_k^+\tau}-e^{\lambda_k^-\tau})e_k+\Phi_t(\tau)(\lambda_k^+ e^{\lambda_k^+\tau}-\lambda_k^-e^{\lambda_k^-\tau})e_k]a_k \notag \\
= & -\sum_{k=1}^n\frac{\mu_k^{\gamma}}{(\lambda_k^--\lambda_k^+)\|e_k\|_U}[\Phi_t'(\tau)e^{\lambda_k^+\tau}(1-e^{(\lambda_k^--\lambda_k^+)\tau})e_k+\Phi_t(\tau)(\lambda_k^+ e^{\lambda_k^+\tau}-\lambda_k^-e^{\lambda_k^-\tau})e_k]a_k
\end{align*}
for every $\tau\in(0,t)$. It follows that
\begin{align*}
\|K_2\psi_t'(\tau)\|_U^2
\leq & C\sum_{k=1}^n\mu_k^{2\gamma}(|\tau\Phi_t'(\tau)|^2e^{-2\rho\mu_k^\alpha\tau}+|\Phi_t(\tau)|^2e^{-2\rho\mu_k^\alpha\tau})a_k^2 \leq C t^{-2m-2}\tau^{2m-\frac{2\gamma}{\alpha}}\|G_na\|_{H_n}^2
\end{align*}
for every $\tau\in(0,t)$. Therefore, we obtain
\begin{align*}
\int_0^t\|K_2\psi_t'(\tau)\|_U^2d\tau\leq Ct^{-2m-2}\int_0^t\tau^{2m-\frac{2\gamma}{\alpha}}d\tau \|G_na\|_{H_n}^2
\leq Ct^{-1-\frac{2\gamma}{\alpha}}\|G_na\|_{H_n}^2.
\end{align*}
\end{proof}




\subsubsection{The main result}
The following result shows that, under suitable conditions on the parameters $\alpha,\gamma,\sigma,\theta$ and $\delta$, Hypotheses \ref{hyp:finito-dimensionale} and \ref{hyp:goal-addo} are fulfilled, and so estimate \eqref{lip-scarsa} holds true. The proof of the first statement follows that lines of that of \cite[Theorem 5.1.3]{AddBig}, but we provide the details for reader's convenience. The second statement is identical to that of the quoted theorem, except from the presence of $\sigma$, hence we skip the proof. 
\begin{thm}
\label{thm:damped_main_result_alpha<12}
Assume that:
\begin{enumerate}[\rm(i)]
\item $\alpha\in\left(0,\frac12\right]$ and $\gamma\in\left[0,\frac\alpha2\right)$;
\item $\theta\in\left(\frac{2}{3}   \cdot\frac{\gamma+\alpha}{\alpha},1\right)$ and $C\in C_b^\theta(H;U)$;
\item $\sigma>\frac12-\alpha$ and $\delta>\frac{1}{2\gamma+\alpha}$;
\item the following series is finite:
\begin{align}
\label{damped_cond_path_uniq_1_alpha<12}
\sum_{n\in\N}\mu_n^{-2\sigma-\alpha}\left\|\langle C(\cdot), \frac{e_n}{\|e_n\|_U}\rangle_U\right\|_{C_b^\theta(H)}^2.  
\end{align}
\end{enumerate}
Then, \eqref{lip-scarsa} holds true.
\end{thm}
\begin{proof}
Under the above assumptions, from Proposition \ref{prop:damped_conv_stoc} we infer that Hypothesis \ref{hyp:finito-dimensionale}(ii) is fulfilled. Moreover, from Theorem \ref{thm:stime_controllo_alpha<12} it follows that Hypotheses \ref{hyp:finito-dimensionale}(vii) is verified. We also notice that Hypotheses \ref{hyp:goal-addo}(ii)(a)-(c) are satisfied with $d_n=2$, $e_1^n=\Psi_n^-$ and $e_2^n=\Psi_n^+$ for every $n\in\N$, and Hypothesis \ref{hyp:goal-addo}(i) holds true from Proposition \ref{prop:stima_espl}. It remains to check that also Hypothesis \ref{hyp:goal-addo}(ii)(d) is fulfilled. To this aim, from the definition of $\widetilde G$, $\Psi_n^+$ and $\Psi_n^-$ and recalling that $B=\widetilde G \widetilde B$, we get 
\begin{align*}
\langle B(x),\Psi_n^+\rangle_H=
\langle \begin{pmatrix}
0 \\ \widetilde B(x)    
\end{pmatrix},\begin{pmatrix}
-\mu_n^{\frac12}e_n \\ \lambda_n^+ e_n    
\end{pmatrix}\rangle_H=\lambda_n^+\langle \widetilde B(x),e_n\rangle_U
\end{align*}
and, analogously, $\langle B(x),\Psi_n^-\rangle_H=\chi_n\lambda_n^-\langle \widetilde B(x),e_n\rangle_U$ for every $n\in\N$. From \eqref{damped_stime_coefficienti_avl} and the fact that $\zeta_n\sim \mu_n^{-\sigma}$, condition \eqref{conv_serie_holder} can be written as
\begin{align}
\label{damped_serie_holder}
\sum_{n\in\N}\mu_n^{-2\sigma}\left((\lambda_n^+)^2\|e_n\|_U^2\frac{\|\langle \widetilde B(\cdot),\frac{e_n}{\|e_n\|_U}\rangle_U\|_{C_b^\theta(H)}^2}{\mu_n^{\alpha}}+\chi_n^2(\lambda_n^-)^2\|e_n\|_U^2\frac{\|\langle \widetilde B(\cdot),\frac{e_n}{\|e_n\|_U}\rangle_U\|_{C_b^\theta(H)}^2}{\mu_n^{\alpha}}\right)<\infty.    
\end{align}
From \eqref{damped_stime_coefficienti_avl}, we get $|\lambda_n^+|, |\lambda_n^-|,\|e_n\|_U\sim \mu_n^{-\frac12}$, and $\chi_n\sim {\rm const}(\chi)$ as $n$ goes to $\infty$. Therefore, the series in \eqref{damped_serie_holder}
behaves as 
\begin{align*}
\sum_{n\in\N}\mu_n^{-2\sigma-\alpha}\left\|\langle \widetilde B(\cdot),\frac{e_n}{\|e_n\|_U}\rangle_U\right\|_{C_b^\theta(H)}^2.
\end{align*}
Since there exists a positive constant $c$, independent on $n\in\N$, such that
\begin{align*}
\left\|\langle \widetilde B(\cdot),\frac{e_n}{\|e_n\|_U}\rangle_U\right\|_{C_b^\theta(H)}^2\leq c\left\|\langle C(\cdot),\frac{e_n}{\|e_n\|_U}\rangle_U\right\|_{C_b^\theta(H)}^2, \qquad n\in\N,
\end{align*}
it follows that, if \eqref{damped_cond_path_uniq_1_alpha<12} holds true, then \eqref{damped_serie_holder} is verified.
\end{proof}

\begin{thm}
\label{thm:damped_main_result_alpha>12}
Assume that:
\begin{enumerate}[\rm(i)]
\item $\alpha\in\left[\frac12,1\right)$ and $\gamma\in\left[0,\frac12-\frac\alpha2\right)$;
\item $\theta\in\left(\frac23\cdot \frac{\gamma+1-\alpha}{1-\alpha},1\right)$ if $\gamma+2\alpha<\frac32$, $\theta\in\left(\frac{4\gamma+2\alpha-1}{2\gamma+\alpha},1\right)$ if $\gamma+2\alpha\geq \frac32$ and $C\in C_b^\theta(H;U)$;
\item $\sigma>0$ and $\delta>\frac{1}{2\gamma+\alpha}$;
\item the following series is finite:
\begin{align}
\label{damped_cond_path_uniq_1_alpha>12}
\sum_{n\in\N}\mu_n^{-2\sigma-\alpha}\left\|\langle C(\cdot), \frac{e_n}{\|e_n\|_U}\rangle_U\right\|_{C_b^\theta(H)}^2.   
\end{align}
\end{enumerate}
Then \eqref{lip-scarsa} is verified.
\end{thm}

\begin{rmk}
Taking into account Proposition Theorem \ref{pathwiseuniqueneSS} and Proposition \ref{prop:damped_conv_stoc}, it is possible to drop assumption on $\delta$ and (iv) in Theorems \ref{thm:damped_main_result_alpha<12} and \ref{thm:damped_main_result_alpha>12} by assuming that $\Lambda^{-\gamma}\in\mathcal L_2(U;U)$. In this case, the mild solution of \eqref{damped_stoc_equation} satisfies estimate \eqref{lip-forte}, which is stronger than \eqref{lip-scarsa}.
\end{rmk}

We split the applications of Theorems \ref{thm:damped_main_result_alpha<12} and \ref{thm:damped_main_result_alpha>12} into separate statements, which differ in the operator $\Lambda$ we deal with. In the former, we consider $\Lambda$ as a realization of the Laplace operator, while in the latter $\Lambda$ is a realization  of the bi-Laplace operator.
\begin{coro}
\label{dampedalfa}
Let $U=L^2(0,\pi)$ and let $\Lambda=-\Delta$ be the realization of the Laplace operator in $L^2(0,\pi)$ with homogeneous Dirichlet boundary conditions. Assume that $C\in C^\theta_b(H;U)$ and that one of the following sets of assumptions are verified:
\begin{itemize}
\item[\rm(i)]  $\alpha\in\left(\frac14,\frac12\right]$, $\gamma\in\left(\frac14-\frac\alpha2,\frac\alpha2\right)$, $\theta\in\left(\frac23\cdot\frac{\gamma+\alpha}{\alpha},1\right)$ and $\sigma>\frac12-\alpha$;
\item[\rm(ii)] $\alpha\in\left[\frac12,1\right)$, $\gamma\in\left(\frac14-\frac\alpha2,\frac12-\frac\alpha2\right)\cap[0,\infty)$, $\theta\in\left(\frac23\cdot\frac{\gamma+1-\alpha}{1-\alpha},1\right)$ if $\gamma+2\alpha<\frac32$ and $\theta\in\left(\frac{4\gamma+2\alpha-1}{2\gamma+\alpha},1\right)$ if $\gamma+2\alpha\geq \frac32$, and $\sigma>0$.
\end{itemize}
Therefore, for mild solutions to  \eqref{damped_stoc_equation} formula \eqref{lip-scarsa} holds true. 
\end{coro}
\begin{proof}
Since $\mu_n\sim n^2$, condition $2=\delta>\frac{1}{2\gamma+\alpha}$ reads as $\gamma>\frac{1}{4}-\frac\alpha2$. Further, both the assumptions in (i) and (ii) ensure that conditions (i)-(iii) in Theorems \ref{thm:damped_main_result_alpha<12} and \ref{thm:damped_main_result_alpha>12} are satisfied. Finally, we show that \eqref{damped_cond_path_uniq_1_alpha<12} and \eqref{damped_cond_path_uniq_1_alpha>12} are verified under (i) or (ii).

Assume that the conditions in (i) are satisfied. Since $\|\langle C(x),e_n\|e_n\|_U^{-1}\rangle_U\|^2_{C^\theta_b(H)}\leq \|C\|_{C^\theta_b(H;U)}$ for every $n\in\N$, it follows that a sufficient condition for \eqref{damped_cond_path_uniq_1_alpha<12} to hold is
\begin{align}
\label{damped_serie_solo_mun}
\sum_{n\in\N}\mu_n^{-\alpha-2\sigma}<\infty.    
\end{align}
We notice that $\mu_n^{-\alpha-2\sigma}\sim n^{-2\alpha-4\sigma}$ and, since $\sigma>\frac12-\alpha$ and $\alpha\leq\frac12$, we infer that $-2\alpha-4\sigma<-2+2\alpha\leq -1$. Hence, \eqref{damped_serie_solo_mun} holds true.

Condition \eqref{damped_serie_solo_mun} is sufficient also when (ii) is verified (see \eqref{damped_cond_path_uniq_1_alpha>12}). In this case, since $\alpha\geq\frac12$ and $\sigma >0$, we conclude that $\mu_n^{-\alpha-2\sigma}\sim n^{-2\alpha-4\sigma}$ is summable.  
\end{proof}

\begin{coro}
\label{coro:damped_euler_beam}
Let $U=L^2((0,\pi)^m)$ and let $\Lambda=(-\Delta)^2$ be the realization of the bi-Laplace operator in $L^2((0,\pi)^m)$ with homogeneous Dirichlet boundary conditions, with $m=1,2,3$. Assume that $C\in C_b^\theta(H;U)$ and one of the following sets of assumptions in fulfilled: \begin{enumerate}[\rm(i)]
\item $\alpha\in\left(\frac m8,\frac12\right]$, $\gamma\in\left(\frac m8-\frac\alpha2,\frac\alpha2\right)\cap[0,\infty)$, $ \theta\in\left(\frac{2}{3}\cdot\frac{\gamma+\alpha}{\alpha},1\right)$, $\sigma>\frac12-\alpha$;
\item $\alpha\in \left[\frac12,1\right)$, $\gamma\in\left(\frac m8-\frac\alpha2,\frac12-\frac\alpha2\right)\cap[0,\infty)$, $\theta\in\left(\frac23\cdot\frac{\gamma+1-\alpha}{1-\alpha}\right)$ if $\gamma+2\alpha<\frac32$ and $\theta\in\left(\frac{4\gamma+2\alpha-1}{2\gamma+\alpha},1\right)$ if $\gamma+2\alpha\geq \frac32$, $\sigma>0$.
\end{enumerate}
If $m=3$, then we further assume that \eqref{damped_cond_path_uniq_1_alpha<12} or \eqref{damped_cond_path_uniq_1_alpha>12}, respectively, is satisfied. Hence, mild solutions to \eqref{damped_stoc_equation} verify estimate \eqref{lip-scarsa}.
\end{coro}
\begin{proof}
The proof follows the lines of \cite[Corollary 5.11]{AddBig}, hence we left the details to the readers. We simply notice that, in this case, $\delta=\frac4m$, and so conditions (i) or (ii) ensure that assumptions (i)-(iii) in Theorem \ref{thm:damped_main_result_alpha<12} are fulfilled. 

Further, since
\begin{align*}
\mu_n^{-\alpha-2\sigma}\sim n^{-\frac{4\alpha+8\sigma}{m}}, \quad n\in\N,
\end{align*}
it follows that, if $m<3$, then the series in \eqref{damped_cond_path_uniq_1_alpha<12} and in \eqref{damped_cond_path_uniq_1_alpha>12} converge for every choice of $C$, under assumptions (i) or (ii), respectively. On the other hand, if $m=3$ then we should additionally require that the series in \eqref{damped_cond_path_uniq_1_alpha<12} is finite. This happens under additional assumptions on $C$ or, for instance, if $\frac{4\alpha+8\sigma}{3}>1$, which holds true whenever $\sigma>\frac38-\frac\alpha2$.
\end{proof}

\subsection{Stochastic damped wave equation with change of the space}
If the nonlinear term $C$ in \eqref{damped_stoc_equation} only depends on $y$, then we can exploit the approach in \cite[Example 5.8]{Dap-Zab14} and \cite{MasPri17,MasPri23}, where the space $H=H^1_0(0,\pi)\times L^2(0,\pi)$ is replaced by the space $\widetilde H= L^2(0,\pi)\times H^{-1}(0,\pi)$ in order to obtained the well-posedness of the stochastic convolution.\\
In particular, we consider a real separable Hilbert space $U$ and we deal with the nonlinear stochastic damped wave equation of the form
\begin{align}
\label{damped_wave_eq_spazio_new}
\left\{
\begin{array}{ll}
\displaystyle \frac{\partial^2y}{\partial t^2}(t)
= -\Lambda y(t)-\rho\Lambda^{\alpha}\left(\frac{\partial y}{\partial t}(t)\right)+C\left(y(t)\right)+\dot{W}(t), & t\in[0,T], \vspace{1mm} \\
y(0)=y_0, & \vspace{1mm} \\
\displaystyle \frac{\partial y}{\partial t}(0)=y_1,
\end{array}
\right.    
\end{align}
with $\Lambda:D(\Lambda)\subseteq U\to U$ is a positive self-adjoint operator with compact resolvent and simple eigenvalues and $C:U\to U$. For every $\beta\in \R$, we set $U_\beta:=D(\Lambda^{-\beta})={\rm Im}(\Lambda^{\beta})$. \\
We consider the space $H_\xi:=U_{\xi}\times U_{\xi}$ for some $\xi\in\left(0,\frac12\right]$ to be fixed later. We notice that the operator $\Lambda$ can be extended to $U_\xi$ and we denote by $\Lambda_\xi:D(\Lambda_\xi)\subseteq U_\xi\to U_\xi$ such an extension, where $D(\Lambda_\xi)=D(\Lambda^{1-\xi})$. In this new setting, the operators $A,G$ and $\widetilde{\mathscr L}$ and the function $B$ reads as
\begin{align*}
& D(A):= \left\{\begin{pmatrix}
h_1 \\ h_2    
\end{pmatrix}:h_2\in D((\Lambda_\xi)^{\frac12}), \ h_1+\rho(\Lambda_\xi)^{\alpha-\frac12}h_2\in D((\Lambda_\xi)^{\frac12})\right\}, \notag \\ 
& A:=\begin{pmatrix}
0 & (\Lambda_\xi)^{\frac12} \\
-(\Lambda_\xi)^{\frac12} & -\rho (\Lambda_\xi)^{\alpha}
\end{pmatrix}, 
\quad \widetilde G=\begin{pmatrix}
0 \\ {\rm Id}   
\end{pmatrix}, \quad \widetilde{\mathscr L}=\begin{pmatrix}
0 & 0 \\
0 & (\Lambda_\xi)^{-\xi}
\end{pmatrix}, \quad \mathcal V:=(\Lambda_\xi)^{-\xi}, \quad  G=\widetilde G(\Lambda_\xi)^{-\xi}  
\end{align*}
and $B(h)=\widetilde G\widetilde B(h)$  for every $h=\begin{pmatrix}
h_1 \\ h_2    
\end{pmatrix}\in H_\xi$, where $\widetilde B(h)=C_\xi(\Lambda^{-\frac12}h_1)$ for every $h=\begin{pmatrix}
h_1 \\ h_2    
\end{pmatrix}\in H_\xi$ and $C_\xi(z)=(\Lambda_\xi)^{\xi}C(z)\in U_\xi$ for every $z\in U$. Let us notice that $B$ is well-defined as function from $H_\xi$ into $U_\xi$: indeed, for every $u\in U_\xi$ we get $\Lambda^{-\frac12}u=\Lambda^{-\frac12+\xi}\Lambda^{-\xi}u\in U_{\frac12-\xi}\subseteq U$, since $\frac12-\xi\geq0$, and so
\begin{align*}
B(h)=\begin{pmatrix}
0 \\  (\Lambda_\xi)^{\xi}C(\Lambda^{-\frac12}h_1) 
\end{pmatrix} \in H_\xi ,
\qquad h=\begin{pmatrix}
h_1 \\ h_2    
\end{pmatrix}\in H_\xi.
\end{align*}
Moreover, we claim that if $C\in C_b^\theta(U;U)$ for some $\theta\in (0,1)$, then $\widetilde B\in C_b^\theta(H_\xi;U_\xi)$. Indeed, for every $h,k\in H_\xi$ we infer that
\begin{align*}
\|\widetilde B(h)-\widetilde B(k)\|_{U_\xi}
= & \|C_\xi(\Lambda^{-\frac12}h_1)-C_\xi(\Lambda^{-\frac12}k_1)\|_{U_\xi} 
=  \|C(\Lambda^{-\frac12}h_1)-C(\Lambda^{-\frac12}k_1)\|_{U} \\
\leq & \|C\|_{C_b^\theta(U;U)}\|\Lambda^{-\frac12}h_1-\Lambda^{-\frac12}k_1\|_{U}^\theta
\leq  \|C\|_{C_b^\theta(U;U)}\|\Lambda^{-\frac12+\xi}\|_{\mathcal L(U)}^\theta\|\Lambda^{-\xi}(h_1-k_1)\|_{U}^\theta \\
= & \|C\|_{C_b^\theta(U;U)}\|\Lambda^{-\frac12+\xi}\|_{\mathcal L(U)}^\theta\|h_1-k_1\|_{U_\xi}^\theta
\leq \|C\|_{C_b^\theta(U;U)}\|\Lambda^{-\frac12+\xi}\|_{\mathcal L(U)}^\theta\|h-k\|_{H_\xi}^\theta,
\end{align*}
which proves the claim.

It follows that equation \eqref{damped_wave_eq_spazio_new}, in its abstract form given by
\begin{align}
\label{damped_wave_eq_abstr}
dX(t)=AX(t)dt+\widetilde{\mathscr L}B(X(t))dt+GdW(t), \quad t\in[0,T], \qquad X(0)=x\in H_\xi,    
\end{align}
is a special case of the equation in Subsection \ref{sub:stoch_damp_eq}, with $U$, $H$ and $\Lambda$ replaced by $U_\xi$, $H_\xi$ and $\Lambda_\xi$, respectively, and $\sigma=\gamma=\xi$. 

We collect the main abstract results \ref{thm:damped_main_result_alpha<12} and \ref{thm:damped_main_result_alpha>12} in the following statement.

\begin{thm}
\label{thm:damped_main_result_xi}
Assume that one of the following set of assumptions is satisfied:
\begin{enumerate}[\rm(i)]
\item $\alpha\in\left(\frac13,\frac12\right]$,  $\xi\in\left(\frac12-\alpha,\frac\alpha2\right)$ and $\theta\in\left(\frac23\cdot\frac{\xi+\alpha}{\alpha},1\right)$;
\item $\alpha\in\left[\frac12,1\right)$, $\xi\in\left(0,\frac12-\frac\alpha2\right)$, $\theta\in\left(\frac23\cdot \frac{\xi+1-\alpha}{1-\alpha},1\right)$ if $\xi+2\alpha<\frac32$, $\theta\in\left(\frac{4\xi+2\alpha-1}{2\xi+\alpha},1\right)$ if $\xi+2\alpha\geq \frac32$.
\end{enumerate}
assume further that $C\in C_b^\theta(U;U)$, $\delta>\frac{1}{2\xi+\alpha}$ and the following series is finite:
\begin{align}
\label{damped_cond_path_uniq_1_xi}
\sum_{n\in\N}\mu_n^{-2\sigma-\alpha}\left\|\langle C(\cdot), \frac{e_n}{\|e_n\|_U}\rangle_U\right\|_{C_b^\theta(U)}^2.  
\end{align}
Then, \eqref{lip-scarsa} holds true.
\end{thm}

We conclude this subsection with a concrete example, analogous to Corollary \ref{dampedalfa}. In particular, we show that if we consider a nonlinear stochastic damped wave equation with nonlinearity which only depends on $y$, as in equation \eqref{damped_wave_eq_spazio_new}, then for some $\alpha\in(0,1)$ we obtain \eqref{lip-scarsa} without coloring either the drift and the noise. We also provide an example which shows that, in the deterministic case, i.e., with $G=0$, equation \eqref{damped_wave_eq_spazio_new} should be ill-posed.

\begin{coro}
\label{coro:dampedalfa_xi}
Let $U=L^2(0,\pi)$ and let $\Lambda=-\Delta$ be the realization of the Laplace operator in $L^2(0,\pi)$ with homogeneous Dirichlet boundary conditions. Assume that $C\in C^\theta_b(U;U)$ and that one of the following sets of assumptions are verified:
\begin{itemize}
\item[\rm(i)]  $\alpha\in\left(\frac13,\frac12\right]$ and $\theta\in\left(\frac1{3\alpha},1\right)$;
\item[\rm(ii)] $\alpha\in\left[\frac12,1\right)$, $\theta\in\left(\frac23,1\right)$ if $\alpha<\frac34$ and $\theta\in\left(\frac{2\alpha-1}{\alpha},1\right)$ if $\alpha\geq \frac34$.
\end{itemize}
Therefore, for mild solutions to  \eqref{damped_wave_eq_spazio_new} formula \eqref{lip-scarsa} holds true. 
\end{coro}
\begin{proof}
At first, we recall that $\xi=\sigma=\gamma$. Further, since $\delta=2$ then $\delta>\frac{1}{2\xi+\alpha}$ is equivalent to $\xi>\frac14-\frac\alpha2=\frac12\left(\frac12-\alpha\right)$, which is always verified if $\alpha\geq\frac12$ since $\xi$ will be chosen greater than $0$.

Assume that (i) is satisfied. Notice that if $\theta\in\left(\frac1{3\alpha},1\right)$, then $\frac12<\frac32\theta\alpha$ and $\frac32\theta\alpha-\alpha<\frac\alpha2$. Hence, if we choose $\xi\in\left(\frac12-\alpha,\frac32\theta\alpha-\alpha\right)$, then $\xi>\frac14-\frac\alpha2$ and $\theta\in\left(\frac23\cdot\frac{\xi+\alpha}{\alpha},1\right)$. It follows that condition (i) in Corollary \ref{dampedalfa} are fulfilled and this implies that \eqref{lip-scarsa} holds true with $H$ replaced by $H_\xi$.

If condition (ii) is verified, we separately consider the cases $\alpha<\frac34$ and $\alpha\geq \frac34$. \\ If $\alpha\in \left[\frac12,\frac34\right)$, which implies that $\theta\in\left(\frac23,1\right)$, then $0<\left(\frac32\theta-1\right)(1-\alpha)<\frac12-\frac\alpha2$ and $\frac32-2\alpha>0$. Therefore, the choice $\xi\in\left(0,\left(\frac32\theta(1-\alpha)-(1-\alpha)\right)\wedge\left(\frac32-2\alpha\right)\right)\subseteq \left(0,\frac12-\frac\alpha2\right)$ gives $\xi+2\alpha<\frac32$ and $\theta\in\left(\frac23\cdot\frac{\xi+1-\alpha}{1-\alpha},1\right)$. \\
If $\alpha\in\left[\frac34,1\right)$, then $2\alpha\geq \frac32$ and $0<\frac{1}{2(2-\theta)}-\frac\alpha2<\frac12-\frac\alpha2$. Hence, if we take $\xi\in\left(0,\frac{1}{2(2-\theta)}-\frac\alpha2\right)\subseteq \left(0,\frac12-\frac\alpha2\right)$, then $\xi+2\alpha\geq \frac32$ and $\theta\in\left(\frac{4\xi+2\alpha-1}{2\xi+\alpha},1\right)$. \\
In both the cases, condition (ii) in Corollary \ref{dampedalfa} is fulfilled and therefore we obtain \eqref{lip-scarsa} in the space $H_\xi$.
\end{proof}
\begin{rmk}
We stress that, even if we change the space where we look for a solution, the function $y$ is a true function and not a distribution, since $y(t)=\Lambda^{-\frac12}X_1(t)\in U=L^2(0,\pi)$ for every $t\in[0,T]$.     
\end{rmk}

The following result has been obtained in \cite[Subsection 5.1.4]{AddBig} and shows that the deterministic counterpart of equation \eqref{damped_wave_eq_spazio_new} should be ill-posed.
\begin{coro}
In $U=L^2(0,\pi)$ we consider the semilinear damped wave equation
\begin{align}
\label{Count_det_damped_wave_eq_1}
\left\{
\begin{array}{lll}
\displaystyle \frac{\partial^2y}{\partial\tau^2}(\tau,\xi)
= \frac{\partial^2y}{\partial\xi^2}(\tau,\xi)-\left(-\frac{\partial^2}{\partial_\xi^2}\right)^{\frac7{12}} \!\frac{\partial y}{\partial\tau} (\tau,\xi)+c(\xi,y(\tau,\xi)), & \xi\in [0,\pi], & \tau\in[0,1],  \\
y(\tau,0)=y(\tau,\pi)=0, & & \tau\in[0,1], \\
\displaystyle y(0,\xi)=\frac{\partial y}{\partial \xi}(0,\xi)=0, & \xi \in [0,\pi],
\end{array}
\right.
\end{align}
where, for every $\xi\in[0,\pi]$ and $y\in\R$,
\begin{align*}
c(\xi,y)
=  & \varphi(y)\left(56({\rm sgn}(\sin(2\xi)))|\sin(2\xi)|^{\frac14}|y|^{\frac34} 
+8\cdot 4^{\frac7{12}}({\rm sgn}(\sin(2\xi)))|\sin(2\xi)|^{\frac18}|y|^{\frac7{8}} +4y\right)
\end{align*}
and $\varphi\in C_c^\infty(\R)$ satisfies $0\leq \varphi\leq 1$, $\varphi\equiv1$ in $(-2,2)$ and $\varphi\equiv0$ in $(-3,3)^c$. Then, both $y_1(\tau,\xi)=0$ and $y_2(\tau,\xi)=\tau^8\sin(2\xi)$ for every $(t,\xi)\in[0,T]$ are solutions to \eqref{Count_det_damped_wave_eq_1}.
\end{coro}

\section*{Statements and Declarations}

\textbf{Competing Interests} Authors are required to disclose financial or non-financial interests that are directly or indirectly related to the work submitted for publication. Please refer to “Competing Interests and Funding” below for more information on how to complete this section.\\

\textbf{Fundings} The authors are members of GNAMPA (Gruppo Nazionale per l'Analisi Matematica,
la Probabilit\`a e le loro Applicazioni) of the Italian Istituto Nazionale di Alta Matematica (INdAM). D. Addona has been partially supported by the INdaMGNAMPA project 2023 ``Operatori ellittici vettoriali a coefficienti illimitati in spazi $L^p$'' CUP E53C22001930001. 
D. A. Bignamini has been partially supported by the INdaMGNAMPA
Project 2023 ``Equazioni differenziali stocastiche e operatori di Kolmogorov in dimensione infinita'' CUP E53C22001930001. 
The authors have no relevant financial or non-financial interests to disclose.\\

\textbf{Acknowledgments} The authors are really grateful to Enrico Priola for the inspiring discussions about weak existence of a mild solution to \eqref{SDE}.

\end{document}